\documentclass[11pt, a4paper, reqno]{amsart}

\usepackage[utf8]{inputenc}
\usepackage[T1]{fontenc}
\usepackage{amsmath, mathtools, amsthm}
\usepackage{amssymb,url,xspace}
\usepackage[mathscr]{eucal}
\usepackage{graphicx}
\usepackage[outercaption]{sidecap} 
\usepackage{subfigure}
\usepackage{bm}
\usepackage[dvipsnames]{xcolor}
\usepackage{cases}
\usepackage[hidelinks]{hyperref}
\usepackage{cleveref}
\usepackage{tikz}
\usepackage{verbatim}
\usetikzlibrary{matrix}
\usepackage{xfrac}

\usepackage[outercaption]{sidecap} 

\usepackage{subfigure}
\usepackage{enumitem}
\usepackage{multirow}
\usepackage[ruled,vlined]{algorithm2e}
\usetikzlibrary{positioning,decorations.pathreplacing}
\usepackage[square]{natbib} 

\definecolor{dukeblue}{rgb}{0.0, 0.0, 0.61}
\definecolor{darkcandyapplered}{rgb}{0.64, 0.0, 0.0}

\hypersetup{
	colorlinks,
	citecolor=dukeblue,
	filecolor=black,
	linkcolor=dukeblue,
	urlcolor=black
}
\usepackage{empheq}

\usepackage[svgnames,pdf]{pstricks}
\usepackage{anysize}
\marginsize{3.25cm}{3.25cm}{3.25cm}{3.25cm}

\newcommand{\diff}{\, \mathrm{d}}

\newcommand{\del}{\partial}
\newcommand{\N}{\mathbb{N}}
\newcommand{\R}{\mathbb{R}}

\def\*#1{\mathbf{#1}}

\theoremstyle{plain}
\numberwithin{equation}{section}
\newtheorem{remark}{Remark}

\newtheorem{proposition}{Proposition}[section]

\newtheorem{lemma}{Lemma}[section]

\newtheorem{example}{Example}[section]

	\title 
	[Optimal actuator design via Brunovsky's normal form]{Optimal actuator design via Brunovsky's normal form}
	
		\author{Borjan Geshkovski}
	
	\address {\textbf{\textup{Borjan Geshkovski}}
	\newline \indent
	{\textup{Departamento de Matem\'aticas}
	\newline \indent
	\textup{Universidad Aut\'onoma de Madrid}}
	\newline \indent
	\textup{28049 Madrid, Spain} 
	\newline \indent \hspace{2.5cm} \textit{and}	
	\newline \indent 
	\textup{Chair of Computational Mathematics} \hspace{1.28cm}
	\newline \indent
	\textup{Fundaci\'on Deusto}
	\newline \indent
	\textup{Av. de las Universidades, 24}
	\newline \indent
	\textup{48007 Bilbao, Basque Country, Spain} 
	}
	\email{\href{mailto:borjan.geshkovski@uam.es}{\textcolor{dukeblue}{\texttt{borjan.geshkovski@uam.es}} }}
	
	\author{Enrique Zuazua}
	\address{\textbf{\textup{Enrique Zuazua}}
		\newline \indent
		\textup{Chair in Applied Analysis, Alexander von Humboldt-Professorship}
		\newline \indent
		\textup{Department of Mathematics} \newline \indent
		\textup{Friedrich-Alexander-Universit\"at Erlangen-N\"urnberg}
		\newline \indent
		\textup{91058 Erlangen, Germany}
		\newline \indent \hspace{2.5cm} \textit{and} \newline \indent
	\textup{Chair of Computational Mathematics} 
	\newline \indent
	\textup{Fundaci\'on Deusto}
	\newline \indent
	\textup{Av. de las Universidades, 24}
	\newline \indent
	\textup{48007 Bilbao, Basque Country, Spain}
	\newline \indent \hspace{2.5cm} \textit{and} \newline \indent
	\textup{Departamento de Matemáticas} \newline \indent
		\textup{Universidad Autónoma de Madrid}
	\newline \indent
	\textup{28049 Madrid, Spain}
	}
	\email{\href{mailto:enrique.zuazua@fau.de}{\textcolor{dukeblue}{\texttt{enrique.zuazua@fau.de}}}}
		
\date{\today}

\begin{document}
	
		\begin{abstract}
		In this paper, by using the Brunovsky normal form, we provide a reformulation of the problem consisting in finding the actuator design which minimizes the controllability cost for finite-dimensional linear systems with scalar controls.
Such systems may be seen as spatially discretized linear partial differential equations with lumped controls. 
		The change of coordinates induced by Brunovsky's normal form allows us to remove the restriction of having to work with diagonalizable system dynamics, and does not entail a randomization procedure as done in past literature on diffusion equations or waves. 
		Instead, the optimization problem reduces to a minimization of the norm of the inverse of a change of basis matrix, and allows for an easy deduction of  existence of solutions, and for a clearer picture of some of the problem's intrinsic symmetries. 
		Numerical experiments help to visualize these artifacts, indicate further open problems, and also show a possible obstruction of using gradient-based algorithms -- this is alleviated by using an evolutionary algorithm.
		\end{abstract}
			
	\maketitle	
	
	\setcounter{tocdepth}{1}
	
	\tableofcontents
	
	\smallskip

	{\small {\href{https://mathscinet.ams.org/msc/msc2010.html}{{\bf 	\color{dukeblue}{AMS Subject Classification}}}}. 93B05, 93B60, 	90C26, 34H05.}
	
\section{Introduction}
Due to their importance in many engineering applications, optimal design problems consisting in finding the location wherein a control of least amplitude actuates and ensures the controllability of the underlying system have been investigated in several works over the past decades, in of both the finite and infinite dimensional dynamical systems context. 
The simplest setting in which one can formulate the fundamental problem is that of finite-dimensional linear systems with scalar controls:
	\begin{equation} \label{eq: Ab.overload}
	\boxed{
	\begin{dcases}
	y'(t) - Ay(t) = bu(t) &\text{ in } (0,T), \\
	y(0) = y_0,
	\end{dcases}
	}
	\end{equation}
	where $A\in\mathcal{M}_{n\times n}(\R)$ and $b\in\R^{n}$. 
	Let us assume that $(A, b)$ is controllable, namely, that the Kalman rank condition is satisfied:
	\begin{equation} \label{eq: kalman.rank.condition}
	\text{span}\big\{b,Ab, \ldots, A^{n-1}b\big\} = \R^n.
	\end{equation}
	Now, it is well-known (see \citep{zuazua2007controllability}) that the control $u(\cdot)$ of minimal $L^2(0,T)$--norm steering \eqref{eq: Ab.overload} to $0$ in any given time $T>0$ satisfies
	\begin{equation} \label{eq: control.cost.est.overload}
	\|u\|_{L^2(0,T)} \leqslant \mathfrak{C}(b,T) \left\|y_0\right\|
	\end{equation}
	for some constant $\mathfrak{C}(b,T)>0$ (which also depends on the dynamics $A$) and for all $y_0\in\R^n$. So, for fixed $T>0$, by denoting
	\begin{align*}
	\mathfrak{C}^*(b,T) &:= \inf\big\{\mathfrak{C}(b,T)>0 \colon \eqref{eq: control.cost.est.overload} \text{ holds}\big\} = \inf_{\|y_0\|=1} \|\Gamma_b(y_0)\|_{L^2(0,T)},
	\end{align*}
	where $\R^n\ni y_0\mapsto\Gamma_b(y_0) = u\in L^2(0,T)$ is the "datum to minimal $L^2$--norm control" operator, the problem consisting of finding an actuator $b$ which minimizes the cost of control may be formulated as 
	\begin{equation} \label{eq: inf.control.cost.b}
	\boxed{
	\min_{b\in\mathbb{S}^{n-1}} \mathfrak{C}^*(b,T).
	}
	\end{equation}	
	As it is often done in control theory, looking at problems from the perspective of the \emph{adjoint} may be more illustrative. 
	We recall that \eqref{eq: Ab.overload} is controllable if and only if the adjoint system 
	\begin{equation}
	\begin{dcases}
	p'(t) + A^\top p(t) = 0 &\text{ in } (0,T),\\
	p(T) = p_T,
	\end{dcases}
	\end{equation}
	is observable in any time $T>0$, in the sense that there exists a constant $\mathfrak{C}_T(b)>0$ such that
	\begin{equation} \label{eq: Ab.observability}
	\mathfrak{C}_T(b)\, \|p(0)\|^2 \leqslant \int_0^T \big|\langle b, p(t)\rangle\big|^2 \diff t
	\end{equation}
	holds for all $p_T\in\R^n$. 
	If we assume assume that $A^\top$ is diagonalizable, i.e., it admits a sequence of eigenvalues $\{\lambda_1,\ldots,\lambda_n\}$ and an associated sequence of eigenvectors $\{\Psi_1,\ldots,\Psi_n\}$ forming an orthonormal basis of $\R^n$, we may rewrite the smallest observability constant $\mathfrak{C}_T^*(b)$ by using separation of variables. 
	Indeed, since 
	\begin{equation*}
	p(t) = \sum_{j=1}^n a_j e^{-\lambda_j(T-t)}\Psi_j \hspace{0.5cm} \text{ for } t\in[0,T],
	\end{equation*}
	where $a_j:=\langle p_T, \Psi_j\rangle$, and setting $c_j := a_j e^{-\lambda_j T}$, it may readily be seen that the smallest constant $\mathfrak{C}^*_T(b)>0$ such that \eqref{eq: Ab.observability} holds can be written as
	\begin{align*}
	\mathfrak{C}^*_T(b) &= \inf_{\sum_{j=1}^{n}|c_j|^2=1} \int_0^T \left| \sum_{j=1}^n c_j e^{\lambda_j t}\langle b,\Psi_j\rangle \right|^2 \diff t \\
	&= \inf_{\sum_{j=1}^{n}|c_j|^2=1}\left(\sum_{j=1}^n c_j^2 \big|\langle b,\Psi_j\rangle\big|^2 \frac{e^{2\lambda_j T}-1}{e^{2\lambda_j}} + 2 \sum_{j=1}^n \sum_{k=1}^{j-1} c_k c_j  \langle b,\Psi_j\rangle \langle b,\Psi_k\rangle \frac{e^{2(\lambda_j+\lambda_k)T}-1}{2(\lambda_j+\lambda_k)}\right). 
	\end{align*}
	However, there is no direct way to simplify the above identity -- due to the appearance of cross terms when expanding the square -- without making specific assumptions on the coefficients $c_j$ of the initial data (e.g., by randomizing them as done in previous literature, as discussed in a subsequent section).  
		
	\subsection{Our contributions}
	
	The goal of this work is to rewrite \eqref{eq: inf.control.cost.b} in a problem which is more tractable from both an analytical and computational perspective, and does not require 1). the system to be diagonalizable, or 2). a randomization procedure of the Fourier coefficients of the initial data. 
	We do so by leveraging the finite-dimensional and scalar control structure. Namely, 
\begin{itemize}
	\item By using the Brunovsky normal form (\citep{brunovsky1970classification}, see Lemma \ref{lem: brunovski}), we discover that we can rewrite \eqref{eq: inf.control.cost.b} as a minimization problem for the norm of the inverse of a change of basis matrix. In particular, the cost $\mathfrak{C}^*(T,b)$ can be written as the tensor product of a function of $T$ and another function of $b$. Hence, any optimal actuator $b^*$ is independent of the time horizon $T$.
	 See Proposition \ref{prop: brunovsky.reform}. 
	\smallskip 
	
	\item We further rewrite the reformulated minimization problem as a maximization of the smallest eigenvalue of a related, symmetric and positive definite matrix. (See Lemma \ref{lem: spectral.reform}.) This variational formulation allows us to ensure the existence of solutions (see Proposition \ref{prop: existence}) and also an invariance of the cost with respect to orthogonal transformations which commute with the system dynamics $A$ (see Proposition \ref{prop: invariants}). The latter, in turn, entails non-uniqueness in some cases. 
	\smallskip
	
	\item Finally, in Section \ref{sec: numerics}, we present numerical experiments on three different examples (in low dimensions) to illustrate the insinuated artifacts and stimulate prospective directions and open problems.
	\end{itemize}
	
	\begin{remark}
	Note that since the Kalman rank condition is equivalent to having\footnote{This fact follows by a unique continuation argument, see e.g. \citep[Section 1.5]{tucsnak2009observation} for more detail (where this property is referred to as the Hautus test); see also \citep[Lem. 1]{beauchard2011large} where this test referred to as the Shizuta-Kawashima (SK) condition is used in the context of hypocoercivity.} $\langle b,\Psi\rangle \neq 0$ for all  $\Psi: \, A\Psi = \lambda \Psi$, the functional is nontrivial. 
	\end{remark}
	
	\subsection{Background}
	
	Actuator optimization problems such as the one studied in this work can be formulated easily for a wide variety of finite and infinite dimensional control systems. In particular, such problems are the motivation of a series of works by Privat, Trélat, and Zuazua \citep{privat2013optimal-b,privat2013optimal-a, privat2015optimal, privat2016optimal, privat2017actuator, privat2019spectral}. (See also \citep{gimperlein2017deterministic, bergounioux2019position} for subsequent studies, \citep{trelat2018optimal} for a concise presentation, and \citep{morris2010linear, kalise2018optimal} for problems with fixed initial data.) In these works, Privat, Trélat, and Zuazua consider the setting of linear partial differential equations (typically diffusion equations or waves) -- to illustrate their approach, let us consider the adjoint heat equation
	\begin{equation} \label{eq: adjoint.heat.eq}
	\begin{dcases}
	-p_t - \Delta p = 0 &\text{ in } \Omega \times (0,T), \\
	p = 0 &\text{ in } \del \Omega \times (0,T), \\
	p|_{t=T} = p_T &\text{ in } \Omega,
	\end{dcases}
	\end{equation}
	where $\Omega\subset\R^d$. 
	Equation \eqref{eq: adjoint.heat.eq} is observable in any time $T>0$ and from any open and non-empty subset $\omega\subset\Omega$ in the sense that the observability inequality
	\begin{equation} \label{eq: obs.ineq}
	\mathfrak{C}_T({\omega}) \|p(0)\|_{L^2(\Omega)}^2 \leqslant \int_0^T \int_\omega |p(t,x)|^2 \diff x \diff t 
	\end{equation}
	for some constant $\mathfrak{C}_T(\omega)>0$ and for all $p_T\in L^2(\Omega)$. In this setting, the dual and equivalent problem to optimal actuator design, is that of optimal sensor placement, which consisting in answering:
	\emph{What is the domain $\omega^*\subset\Omega$ with $|\omega^*|=\gamma$ such that the smallest constant $\mathfrak{C}_T({\omega^*})>0$ for which \eqref{eq: obs.ineq} holds, is minimized?}
	
	The approach of these works involves separation of variables using a basis of eigenfunctions $-\Delta \Psi_j = \lambda_j \Psi_j$. Sticking to the design problem for \eqref{eq: adjoint.heat.eq} -- \eqref{eq: obs.ineq} for ease of presentation, one would decompose the solution of \eqref{eq: adjoint.heat.eq} into this basis as $p(t,x) = \sum_{j=1}^\infty a_j e^{-\lambda_j(T-t)}\Psi_j(x)$. If one defines $b_j := a_je^{-\lambda_j T}$, the shape optimization problem can be addressed by examining
	\begin{align*}
	\mathfrak{C}_T(\omega) &= \inf_{\sum_{j=1}^\infty|b_j|^2=1} \int_0^T \int_\omega \left|\sum_{j=1}^\infty b_j e^{\lambda_j t} \Psi_j(x)\right|^2 \diff x \diff t \\
	&= \inf \sigma\left(\left\{\frac{e^{(\lambda_j+\lambda_k)T}-1}{\lambda_j+\lambda_k}\int_\omega\Psi_j(x)\Psi_k(x)\diff x\right\}_{j,k}\right),
	\end{align*}
	where $\sigma$ denotes the spectrum of the intervening infinite-dimensional, symmetric, and nonnegative matrix. This is a challenging spectral optimization problem since little is known about the mixed terms $\int_\omega \Psi_j(x)\Psi_k(x)\diff x$. 
	Indeed, even in the case of the disk, the restriction of inner products of arbitrary Bessel functions to subsets $\omega\subset\Omega$ cannot be computed explicitly. 
	
	In order to avoid computing these mixed terms, Privat, Trélat and Zuazua replace $\{a_j\}_{j\in\N}$ by a sequence of real-valued random variables $\{\beta_j^\nu a_j\}_{j\in\N, \nu \in \mathcal{X}}$; the random variables $\{\beta_j^\nu\}_{j\in\N, \nu\in\mathcal{X}}$ are independent and identically distributed, of mean $0$ and variance $1$, and have fast decay (e.g., following a Bernouilli distribution). The authors then study the case of an averaged observability constant, in which the mixed terms vanish when expanding the quadratic term: 
	\begin{align*}
	\mathfrak{C}_T^{\mathrm{rand}}(\omega) &= \inf \sigma\left(\left\{\frac{e^{(\lambda_j+\lambda_k)T}-1}{\lambda_j + \lambda_k} \mathbb{E}(\beta_j^\nu \beta_k^\nu) \int_\omega \Psi_j(x)\Psi_k(x)\diff x\right\}_{j,k}\right)\\
	&= \inf\sigma\left(\left\{\frac{e^{2\lambda_jT}-1}{2\lambda_j} \int_\omega \Psi_j(x)^2\diff x\right\}_j\right)\\
	&= \inf_{j\in\N} \frac{e^{2\lambda_jT}-1}{2\lambda_j} \int_\omega \Psi_j(x)^2\diff x.
	\end{align*}
	It is to be noted herein that the randomization hypothesis renders the shape optimization problem significantly more tractable, but of course, with the price that there might be a gap between the deterministic and the randomized problem. 
	Going back to the deterministic problem is thus very challenging, which motivates our approach of reformulating the deterministic control (or observation) cost in a different coordinate system. 
	
	\begin{remark} Note that, formulated as such for linear finite-dimensional systems, \eqref{eq: inf.control.cost.b} does not strictly represent a finite-dimensional, discretized version of localized actuator or sensor problems for partial differential equations (such as \eqref{eq: adjoint.heat.eq}). 
	Rather, whenever $A$ is a numerical discretization of some differential operator in one space dimension (e.g., by finite-differences),  \eqref{eq: inf.control.cost.b} can be seen as finding the optimal controller location for a corresponding \emph{lumped control system}. In the context of the heat equation for instance, this would be 
	\begin{equation}
	\begin{dcases}
	y_t(t,x) - y_{xx}(t,x) = b(x)u(t) &\text{ in } (0,T)\times(0,1),\\
	y(t,0) = y(t,1) = 0 &\text{ in } (0,T),
	\end{dcases}
	\end{equation}
	and $A$ could thus represent the finite-difference Laplacian.
	\end{remark}
	
	\subsection*{Notation} For $n\geqslant2$, we denote $\mathbb{S}^{n-1}:=\{x\in\R^n \colon \|x\|=1\}$, and by $\mathrm{GL}_n(\R)$ the group of invertible matrices. Unless otherwise stated, we denote by $\|\cdot\|$ the standard euclidean ($\ell^2$) norm. 
	
	\section{Reformulation via Brunovsky's normal form} \label{sec: brunovsky}

	We begin our study by motivating and recalling the Brunovsky normal form, as to enhance the clarity of the subsequent results. Consider the $n$-th order linear equation
	\begin{equation}
	\zeta^{(n)}(t) + k_1\zeta^{(n-1)}(t) + \ldots + k_n \zeta(t) = u(t),
	\end{equation}
	with real constant coefficients $\{k_i\}_{i=1}^n$. 
	By setting $z:= \left[\zeta \, \zeta' \, \ldots \zeta^{(n-1)}\right]^\top$, one sees that the above equation is equivalent to the linear system
	\begin{equation} \label{eq: 2}
	z'(t) = \mathfrak{A} z(t) + \mathbf{e}_n u(t),
	\end{equation}
	where $\mathbf{e}_n := [0, \ldots, 0, 1]^\top$ denotes the last vector of the canonical basis of $\R^n$, and
	\begin{equation}
	\mathfrak{A} = \begin{bmatrix} 0& 1& 0& \ldots& 0 \\ 
						     0& 0& 1& & \vdots \\
						     \vdots&  &\ddots& \ddots& 0 \\
						     0& \ldots& 0 & 0 & 1 \\
						     -k_n& \hdots &\hdots &\hdots & -k_1 
				\end{bmatrix}
	\end{equation}
	is a \emph{companion} matrix.
	A natural question that arises is the converse: \emph{When can a constant coefficient linear system 
	\begin{equation}
	y'(t) = Ay(t) + bu(t)
	\end{equation}
	where $A\in \mathcal{M}_{n\times n}(\R)$ and $b\in\R^n$ be transformed to \eqref{eq: 2} via $y=Pz$ for some invertible matrix $P\in\mathcal{M}_{n\times n}(\R)$?} 
	
	Note that, should such a relation hold, then
	\begin{align*}
	z'(t) = \left(P^{-1}y\right)'(t) &= \left(P^{-1}AP\right)\left(P^{-1}y\right)(t) + \left(P^{-1}b\right)u(t)\\
				   &= \left(P^{-1}AP\right)z(t) + \left(P^{-1}b\right)u(t), 
	\end{align*}
	and so we are led to ask if there exists an invertible matrix $P\in\mathcal{M}_{n\times n}(\R)$ such that $P^{-1}AP$ is a companion matrix, and $P^{-1}b=\mathbf{e}_n$.

	To answer such a question, the Brunovsky's normal form comes into play.
	\medskip

	\begin{lemma}[Brunovsky normal form, \citep{brunovsky1970classification}] \label{lem: brunovski}
	Let $A\in\mathcal{M}_{n\times n}$ with $n\geqslant 2$ be given.
	If there exists a vector $b\in\R^n$ such that $(A,b)$ satisfy the Kalman rank condition \eqref{eq: kalman.rank.condition}, then there exists an invertible matrix $P=P(b)\in\mathcal{M}_{n\times n}(\R)$ 
	such that 
	\begin{equation} \label{eq: similar}
	A=P\mathfrak{A}P^{-1} \hspace{0.5cm} \text{ and } \hspace{0.5cm} b=P\mathbf{e}_n,
	\end{equation}
	where $\mathfrak{A}$ is the companion matrix of $A$ defined as
	\begin{equation} \label{eq: y.system.def}
	\mathfrak{A} = \begin{bmatrix} 0& 1& 0& \ldots& 0 \\ 
						     0& 0& 1& & \vdots \\
						     \vdots&  &\ddots& \ddots& 0 \\
						     0& \ldots& 0 & 0 & 1 \\
						     -a_n& \hdots &\hdots &\hdots & -a_1 
				\end{bmatrix},
	\end{equation}
	where $\{a_1, \ldots, a_n\}$ are the coefficients of the characteristic polynomial of $A$
	\begin{align*}
	\det(A-x\mathrm{Id}) &:= \prod_{j=1}^n (x-\lambda_j)^{r_j} = x^n + a_1x^{n-1} + \ldots + a_{n-1} x + a_n= 0.
	\end{align*}
	Moreover, the matrix $P(b)$ ensuring \eqref{eq: similar} is unique, its columns $\{f_1,\ldots,f_n\}$ being given by
	\begin{equation} \label{eq: analy.def.fk}
	f_k = 
	\begin{dcases}
	b & k=n \\
	\left(A^{n-k} + \sum_{j=1}^{n-k} a_{j} A^{n-k-j}\right) b, & 1\leqslant k \leqslant n-1.
	\end{dcases}
	\end{equation}
	Conversely, if there exists an invertible matrix $P\in\mathcal{M}_{n\times n}(\R)$ such that $A=P\mathfrak{A}P^{-1}$, then $(A,b)$, with $b:=P\mathbf{e}_n$, satisfies the Kalman rank condition \eqref{eq: kalman.rank.condition}.
	\end{lemma}
	
	\noindent
	For the sake of completeness and clarity, we provide a proof in the appendix (see also \citep{brunovsky1970classification, trelat2005controle}). The Brunovsky normal form has found great success in a variety of contexts, going as far as  gradient descent convergence for machine learning applications (\citep{hardt2016gradient}). Before proceeding, let us provide some comments.
	
	\begin{remark} A well-known result in linear algebra states that a matrix $A\in \mathcal{M}_{n\times n}(\R)$ is similar to its companion matrix $\mathfrak{A}$ (i.e., there exists a $P\in\mathrm{GL}_n(\R)$ such that $A=P\mathfrak{A}P^{-1}$) if and only if $A$ has a cyclic vector (i.e., there exists some $b\in\R^n$ such that \eqref{eq: kalman.rank.condition} holds) -- see for instance \citep[Theorem 3.3.15]{horn2012matrix}. In fact, one sees that Lemma \ref{lem: brunovski} is nothing but a rewriting of this fact. Furthermore, both conditions are equivalent to $A$ having all of its eigenspaces with dimension $\leqslant1$. Hence, a sufficient condition for a square matrix $A$ to be similar to its companion matrix (or equivalently, to have a cyclic vector) is that it has $n$ distinct eigenvalues. This will be the case for the examples we shall consider; a notable one being the finite-difference discretization of the Dirichlet Laplacian in $1d$, whose eigenvalues are precisely $\lambda_j = -\frac{4}{h^2}\sin^2\left(\frac{\pi j}{2(n+1)}\right)$ for $j=1,\ldots,n$ (see \citep{vichnevetsky1982fourier}).
	\end{remark}

	\noindent
	By virtue of the change of coordinates provided by Brunovsky canonical form, we can obtain the following result which allows us to consider an equivalent, but more explicit representation of the cost to be minimized.
	
	\begin{proposition}[\eqref{eq: inf.control.cost.b} in Brunovsky coordinates] \label{prop: brunovsky.reform}
	Let $A\in \mathcal{M}_{n\times n}(\R)$ with $n\geqslant2$ be given, and suppose that $b\in\R^n$ is such that $(A,b)$ satisfies the Kalman rank condition \eqref{eq: kalman.rank.condition}.
	Then,
	\begin{equation*}
	\mathfrak{C}(b,T) = \Big(\left\|P^{-1}(\cdot)\right\| \otimes \kappa(\cdot)\Big)(b,T) := \left\|P^{-1}(b)\right\| \kappa(T),
	\end{equation*}
	where $\kappa(T)>0$ denotes the cost of controllability for $(\mathfrak{A}, \mathbf{e}_n)$ (and thus depends solely on $A$). 
	
	Consequently, whenever $A$ is similar to its companion matrix $\mathfrak{A}$, problem \eqref{eq: inf.control.cost.b} is equivalent to 
	\begin{equation} \label{eq: 18}
	\boxed{
	\inf_{b \in \mathbb{S}^{n-1}} \left\|P^{-1}(b)\right\|.
	}
	\end{equation}
	\end{proposition}
	
	\begin{remark}
	In other words, one sees that now the cost function is independent of $T$, and hence an optimal design $b^*\in\mathbb{S}^{n-1}$ will be as well. One should avoid confusion in this insight, as clearly a minimal $L^2(0,T)$--norm control will depend on the time horizon since the controllability cost $\mathfrak{C}^*(b,T)$ will too -- the
splitting of time and controller variables does not contradict existing results which ensure that $\kappa(T)$ decays as $T\nearrow\infty$, and explodes like $\gamma T^{-\frac{n+1}{2}}$ with $\gamma=(n-1)!\left(\mathfrak{A}^{n-1} \cdot \mathbf{e}_{n-1}\right)^{-1}$ as $T\searrow0$ (see \citep{seidman1988violent}).	
	\end{remark}

	\begin{proof}[Proof of Proposition \ref{prop: brunovsky.reform}]
	Let us suppose that there exists a vector $b\in\R^n$ such that $(A,b)$ is controllable, i.e., $(A,b)$ satisfies \eqref{eq: kalman.rank.condition}. 
	We consider the system 
	\begin{equation} \label{eq: y.system}
	\begin{dcases}
	z'(t) - \mathfrak{A} z(t) = \mathbf{e}_n u(t) &\text{ in } (0,T), \\
	z(0) = z_0,
	\end{dcases}
	\end{equation}
	which is also controllable.
	Moreover, given any $T>0$, there exists a constant $\kappa(T)>0$ depending only on $T$ and $\mathfrak{A}$ (and thus $A$) such that the minimal $L^2$--norm function $u(\cdot)$ ensuring controllability for \eqref{eq: kalman.rank.condition} satisfies
	\begin{equation} \label{eq: control.cost.brunovski.sys}
	\|u\|_{L^2(0,T)} \leqslant \kappa(T) \left\|z_0\right\|
	\end{equation}
	for all $z_0\in\R^n$. 
	Note that the cost of control $\kappa(T)>0$, defined as the smallest constant appearing in \eqref{eq: control.cost.brunovski.sys}, is a priori independent of $b\in\R^n$, since the companion matrix $\mathfrak{A}$ is itself independent of $b$ and depends only on $A$ via its characteristic polynomial. 
	Since $\mathfrak{A} = P^{-1}AP$ and $\mathbf{e}_n = P^{-1}b$, we see that
	\begin{equation}
	z'(t) - P^{-1} A P z(t) = P^{-1}b\, u(t) \hspace{1cm} \text{ for } t\in(0,T),
	\end{equation}
	and multiplying by $P$ to the left, we obtain
	\begin{equation}
	(P z)'(t) - A (Pz)(t)  = bu(t) \hspace{1cm} \text{ for } t\in(0,T).
	\end{equation}
	Therefore, with $y=Pz$, we recover the system \eqref{eq: Ab.overload} from \eqref{eq: y.system} -- \eqref{eq: y.system.def}.
	By virtue of the above computations, and \eqref{eq: control.cost.brunovski.sys}, we deduce
	\begin{align*}
	\|u\|_{L^2(0,T)} \leqslant \kappa(T) \left\|z_0\right\| &= \kappa(T) \left\|P^{-1}y_0\right\|\nonumber \\
	&\leqslant \kappa(T) \left\|P^{-1}\right\| \left\|y_0\right\|.
	\end{align*}
	This bound is sharp, as the cost of control of the original system \eqref{eq: Ab.overload} is precisely
	\begin{equation*}
	\mathfrak{C}(b,T) := \kappa(T) \left\|P^{-1}(b)\right\|.
	\end{equation*}
	This concludes the proof.
	\end{proof}
	
	\noindent
	In other words, the transformation induced by writing the Brunovsky normal form of the original system \eqref{eq: Ab.overload} has allowed to perform a separation of variables of the control cost. Hence, the problem of choosing the controller $b\in\mathbb{S}^{n-1}$ so that the cost of control of \eqref{eq: Ab.overload} is optimized, i.e. \eqref{eq: inf.control.cost.b}, can be reformulated to the problem of optimizing the norm of the inverse of change-of-basis matrix $P(b)$.
	
	\subsection{Computing the norm of $P^{-1}(b)$} 
	As we have seen in what precedes, provided there exists $b\in\R^n$ such that $(A,b)$ satisfies the Kalman rank condition, the change-of-basis matrix
	$
	P(b) \in \mathrm{GL}_n(\R)
	$
	is fully determined out of the coefficients of the characteristic polynomial of $A$, and the value of $b$.
	It would however be convenient to have a simplified description of the norm of $P^{-1}(b)$. 
	The norm which canonically appears in \eqref{eq: 18} is the  standard operator norm, namely $\left\|P^{-1}(b)\right\| := \sup_{\|x\|=1} \left\|P^{-1}(b)x\right\|$ (where the underlying norm is the euclidean one), which could be defined as the largest eigenvalue of an associated symmetric and positive definite matrix, and hence avoids computing the inverse.
	
	In fact, one has the following characterization.
	
	\begin{lemma}[Variational form]
	\label{lem: spectral.reform}
	Suppose that $A\in\mathcal{M}_{n\times n}$ with $n\geqslant 2$ is similar to its companion matrix.
	Problem \eqref{eq: 18} is then equivalent to
	\begin{equation} \label{eq: 23}
	\boxed{
	\max_{b\in\mathbb{S}^{n-1}} \lambda_{1}\left(P(b)P(b)^\top\right).
	}
	\end{equation}
	Here $\lambda_{1}(M)$ denotes the smallest eigenvalue of a matrix $M\in\mathcal{M}_{n\times n}(\R)$.
	\end{lemma}
	
	\begin{proof}[Proof of Lemma \ref{lem: spectral.reform}]
	Noting that $(P(b)^{-1})^\top P(b)^{-1}$ is a symmetric and positive definite matrix (by virtue of the Kalman rank condition, which holds due to the equivalence with $A$ being similar to its companion matrix), it thus admits a sequence of  $n$ real eigenvalues $0 < \lambda_1 \leqslant \ldots \leqslant \lambda_n$. Moreover using classical results from linear algebra, we have
	\begin{equation} \label{eq: 2.26}
	\left\|P(b)^{-1}\right\| = \sqrt{\lambda_{n}\Big((P(b)^{-1})^\top P(b)^{-1}\Big)},
	\end{equation}
	and, noting that $\left(P(b)^{-1}\right)^\top = \left(P(b)^\top\right)^{-1}$, we see that
	\begin{align*}
	\left(P(b)^{-1}\right)^\top P(b)^{-1} = \left(P(b)^\top\right)^{-1} P(b)^{-1} = \left(P(b)P(b)^\top\right)^{-1}.
	\end{align*}
	Using once again the symmetry of $P(b)P(b)^\top$, we see that
	\begin{equation} \label{eq: 2.28}
	\lambda_{n}\Big((P(b)P(b)^\top)^{-1}\Big) = \frac{1}{\lambda_{1}(P(b)P(b)^\top)}.
	\end{equation}
	Accordingly, by positivity and the convexity of the square root, the optimisation problem \eqref{eq: 18} is equivalent to \eqref{eq: 23}.
	\end{proof}
	
	\begin{remark}
	We may, for instance, also consider an explicit representation of the inverse of $P^{-1}(b)$ by the Cayley-Hamilton formula
	\begin{align*}
	P^{-1}(b) = {\frac {1}{\det(P(b))}}\sum _{s=0}^{n-1}P(b)^{s} \sum _{k_{1},k_{2},\ldots ,k_{n-1}}\prod _{l=1}^{n-1}{\frac {(-1)^{k_{l}+1}}{l^{k_{l}}k_{l}!}}\mathrm{trace}(P^{l}(b))^{k_{l}},
	\end{align*}
	where $k_l\geqslant0$ solve the linear Diophantine equation $\displaystyle s+\sum _{l=1}^{n-1}lk_{l}=n-1$, and consider the Frobenius norm instead of the standard operator norm in \eqref{eq: 18}. 
	 Such a formulation is however not all too appealing for numerical purposes due to the implicit need to solve a Diophantine equation in each iteration of the minimization algorithm.
	
	Another way to characterize the inverse could be by using the Cramer formula, but this becomes difficult to track when $n\geqslant3$ due to the involved form of the minors composing the adjunct matrix. In any case, such explicit formulas for the inverse of $P(b)$ appear quite convoluted and difficult to use for a further analysis.
	\end{remark}
	 
	 \noindent
	In view of the equivalent characterization of \eqref{eq: 18} given by \eqref{eq: 23}, and the well-known continuity results for eigenvalues with respect to parameters whenever the underlying matrix possesses such continuity\footnote{All eigenvalues of a matrix $M(t)$ are continuous functions of $t$ whenever the entries of $M(t)$ are continuous functions of $t$. This fact holds whether or not $M(\cdot)$ is invertible and/or positive definite (see e.g., \citep[pp. 116]{kato2013perturbation}).}, we may deduce the following result. 
	
	\begin{proposition} \label{prop: existence}
	Suppose that $A\in\mathcal{M}_{n\times n}(\R)$ with $n\geqslant2$ is similar to its companion matrix. 
	Then, both problems \eqref{eq: 18} and \eqref{eq: 23} admit a solution $b^* \in \mathbb{S}^{n-1}$.
	\end{proposition}
	
	\noindent
	This result is a priori not evident when looking at the equivalent problem of minimizing the norm of the inverse of $P(b)$, but follows as a direct corollary.
	
	\section{Symmetries}
	
	A question which merits asking however, and which does not seem that obvious at first glance since it is not quite clear how one may study the convexity of $b\longmapsto P^{-1}(b)$ (or concavity of $b\longmapsto\lambda_1\left(P(b)P(b)^\top\right)$), is that of uniqueness of minimizers (or the lack thereof). 
	There is no reason as to why one may expect uniqueness. In fact, we prove the following result, which stipulates an invariance of the functional with respect to orthogonal transformations which commute with the system dynamics $A$. 
	
	\begin{proposition}[Invariants] \label{prop: invariants}
	Let $A\in\mathcal{M}_{n\times n}(\R)$ with $n\geqslant2$ be similar to its companion matrix, and let $\*R\in \mathcal{M}_{n\times n}(\R)$ be such that 
	\begin{enumerate}
	\item $[A, \*R]=A\*R-\*RA=0$ (i.e. $A$ and $\*R$ commute);
	\item $\*R\in \mathcal{M}_{n\times n}$ is orthogonal, meaning that $\*R\*R^\top=\*R^\top \*R=\mathrm{Id}_n$.
	\end{enumerate}
	Then we have that
	\begin{equation}
	\min_{b\in\mathbb{S}^{n-1}}\left\|P^{-1}(\*Rb)\right\|^2 = \min_{b\in\mathbb{S}^{n-1}}\left\|P^{-1}(b)\right\|^2.
	\end{equation}
	\end{proposition}	
	\medskip
	
	\noindent
	In other words, provided a minimizer $b^*$, one may, provided commutativity with $A$, rotate $b^*$ to obtain another minimizer $\*Rb^*$.
	
	For example, as seen in the numerical experiments in the following section, the finite-difference Dirichlet Laplacian in $n=2$:
	\begin{equation*}
	\begin{bmatrix}
	-2 & 1\\
	1 & -2
	\end{bmatrix},
	\end{equation*}
	commutes with the orthogonal matrices
	\begin{equation*}
	\begin{bmatrix}-1&0\\0&-1\end{bmatrix}, \begin{bmatrix}0&1\\1&0\end{bmatrix}, \begin{bmatrix}0&-1\\-1&0\end{bmatrix}.
	\end{equation*}
	
	\begin{proof}[Proof of Proposition \ref{prop: invariants}]
	We will make use of the characterization \eqref{eq: 2.26} -- \eqref{eq: 2.28} of the spectral norm of $P^{-1}(\cdot)$. In other words, we recall that since $P(\cdot)P(\cdot)^\top$ is a symmetric and positive definite matrix, we have that
	\begin{equation}
	\left\|P^{-1}(\cdot)\right\|^2 = \frac{1}{\lambda_1\left(P(\cdot)P(\cdot)^\top\right)},
	\end{equation}
	where $\lambda_1\left(P(\cdot)P(\cdot)^\top\right)$ denotes the smallest eigenvalue of $P(\cdot)P(\cdot)^\top$.
	Let us thus concentrate on investigating the invariance properties of $\lambda_1$. 
	
	Let $b\in\R^n$ be fixed. We recall that by the Rayleigh's min-max theorem, we have
	\begin{equation*}
	\lambda_1\left(P(b)P(b)^\top\right) := 
	 \min_{x \in \R^n \setminus\{0\}} \frac{\langle P(b)P(b)^\top x, x\rangle}{\|x\|^2}.
	 \end{equation*}
	 On another hand, making use of \eqref{eq: analy.def.fk}, we may see that  
	\begin{equation} \label{eq: P.rectangular.def}
	P(b) = \underbrace{\big[p_1(A)\,\, \ldots\,\, p_n(A)\big]}_{\in \mathcal{M}_{n\times n^2}(\R)} \underbrace{\begin{bmatrix} b & &\\
	 &\ddots & \\
	 & & b
	 \end{bmatrix}}_{\in\mathcal{M}_{n^2\times n}(\R)},
	\end{equation}
	where 
	\begin{equation}
	p_k(A) := \begin{dcases}A^{n-k} + \sum_{j=1}^{n-k} a_{j} A^{n-k-j} &\text{ for } k\leqslant n-1,\\
	\text{Id} &\text{ for } k=n.
	\end{dcases}
	\end{equation}
	After some computations using \eqref{eq: P.rectangular.def}, we can deduce that
	\begin{equation} \label{eq: pp*.char}
	P(b)P(b)^\top = \sum_{k=1}^n p_k(A)bb^\top p_k(A)^\top.
	\end{equation}
	The above representation combined with the Rayleigh quotient characterization yield
	\begin{align*}
	\lambda_1\left(P(b)P(b)^\top\right) &:= \min_{x \in \R^n \setminus\{0\}} \sum_{k=1}^n \frac{\left\langle p_k(A)bb^\top p_k(A)^\top x, x\right\rangle}{\|x\|^2}\\
	 &=  \min_{x \in \R^n \setminus\{0\}} \sum_{k=1}^n \frac{\left\langle b^\top p_k(A)^\top x, b^\top p_k(A)^\top x\right\rangle}{\|x\|^2} \\
	 &= \min_{x \in \R^n \setminus\{0\}} \sum_{k=1}^n \frac{\left\|(p_k(A)b)^\top x\right\|^2}{\|x\|^2}.
	\end{align*}
	Now since $[A, \*R]=0$ we clearly also have $[p_k(A), \*R]=0$ for $k\leqslant n$. Whence for $x\in\R^n$,
	\begin{align*}
	\left\|(p_k(A)\*Rb)^\top x \right\|^2 &= \left\|(\*Rp_k(A)b)^\top x \right\|^2 = \left\|(p_k(A)b)^\top \*R^\top x \right\|^2
	\end{align*}
	holds. Since $\*R^\top$ is orthogonal, 
	\begin{equation*}
	\frac{\left\|(p_k(A)\*Rb)^\top x \right\|^2}{\|x\|^2} =  \frac{\left\|(p_k(A)b)^\top \*R^\top x \right\|^2}{\|\*R^\top x\|^2}.
	\end{equation*}
	Clearly, since $\*R$ is invertible, 
	\begin{align*}
	&\min_{y \in \R^n \setminus\{0\}} \sum_{k=1}^n \frac{\left\|(p_k(A)b)^\top y\right\|^2}{\|y\|^2} = \min_{x \in \R^n \setminus\{0\}} \sum_{k=1}^n   \frac{\left\|(p_k(A)b)^\top \*R^\top x \right\|^2}{\|\*R^\top x\|^2},
	\end{align*}
	whence we may conclude the proof.
	\end{proof}

	\section{Numerical experiments} \label{sec: numerics} 
	
	We henceforth provide a brief numerical study of the optimization problem. We focus on the reformulation provided by \eqref{eq: 23}, which we recall consists in solving
	\begin{equation}
	\max_{b\in\mathbb{S}^{n-1}} \lambda_1\left(P(b)P(b)^\top\right) = \max_{b\in\mathbb{S}^{n-1}} \min_{x\in \R^n\setminus\{0\}} \frac{\left\langle P(b)P(b)^\top x,x\right\rangle}{\|x\|^2}.
	\end{equation}
	We recall the synthetic definition of $P(b)$ and characterization of $P(b)P(b)^\top$ in \eqref{eq: P.rectangular.def} and \eqref{eq: pp*.char}, respectively. Given a matrix $A\in\mathcal{M}_{n\times n}(\R)$ which is similar to its companion matrix, we shall solve numerically the above optimization problem (i.e. find some maximizer $b^*\in\R^n$) by using 

	\begin{itemize}
	\item \textbf{Case $n=2$:} The \texttt{IPOPT} method via \texttt{CasADi} (\citep{andersson2019casadi}) in \texttt{Matlab}.\footnote{see \href{https://github.com/borjanG/optimal.controller}{\textcolor{dukeblue}{\texttt{https://github.com/borjanG/optimal.controller}}}. Experiments were conducted on a personal MacBook Pro laptop (2.4 GHz Quad-Core Intel Core i5, 16GB RAM, Intel Iris Plus Graphics 1536 MB).}
	We make use of the power iteration algorithm to find the smallest eigenvalue of the symmetric, positive-definite matrix $P(b)P(b)^\top$ by a simple spectral shift: we first find the largest eigenvalue $\lambda_{\max}$, and then find the largest eigenvalue of $P(b)P(b)^\top-\lambda_{\max}$; the sum of both resulting eigenvalues yields the desired smallest eigenvalue. We emphasize the necessity of not using a pre-defined routine for computing the eigenvalue, due to the fact that automatic differentiation requires a graph-like object to be able to differentiate and obtain gradients, and traceability with respect to the optimization variable is in general not provided in a pre-defined routine. 
	\smallskip 
	
	\item \textbf{Case $n\geqslant3$:} Due to a lack of convergence of \texttt{IPOPT} for $n\geqslant3$, which could be due to non-concavity, we make use of an evolutionary algorithm\footnote{We thank Emmanuel Trélat for this insight and suggestion.}. Namely, we use the \emph{differential evolution} algorithm implemented in \texttt{SciPy} (\citep{storn1997differential}). (Such obstacles have been encountered -- and bypassed -- by use of a genetic in related works, see \citep{hebrard2003optimal, freitas1999optimizing}.)
	\end{itemize} 
	
	\noindent
	The algorithms suffer from a curse of dimensionality and are, at least for the  examples presented below, providing answers up to $n\leqslant10$ (an optimization run for $n=10$ took around $8h$ on a personal computer). We provide three basic experiments to motivate possible characterizations of optimal solutions depending on the symmetry properties of the system dynamics $A$. 
	
	\begin{remark} The likely cause of the lack of convergence of gradient-based methods is the lack of concavity of the functional $b\mapsto\lambda_1\left(P(b)P(b)^\top\right)$. Let us briefly comment on this artifact.  By using the Rayleigh characterization of $\lambda_1$, we see that to differentiate one needs to inject derivatives inside the $\min$. Formally applying Danskin's theorem (\citep{danskin1966theory}), to differentiate $b\mapsto\lambda_1\left(P(b)P(b)^\top\right)$ it would roughly suffice to differentiate the map $\Psi: b\mapsto \langle Mbb^\top M^\top x, x\rangle$ for fixed $x\in\R^n$, where $M\in\mathcal{M}_{n\times n}(\R)$ is fixed. In essence, this reduces to differentiating the square matrix $bb^\top\in\mathcal{M}_{n\times n}(\R)$ with respect to $b$ -- a first differentiation yields a $3$-tensor $\*D^1\in\R^{n\times n\times n}$ where $\*D^1_{k, j, \ell} = \del_{b_\ell} (bb^\top)_{j, k} = b_j\delta_{\ell,k}+b_k\delta_{\ell,j}$, where $\delta_{j,k}$ denotes the Kronecker delta. A second differentiation would yield a $4$-tensor $\*D^2\in\R^{n\times n \times n \times n}$, where $\*D^2_{j,k,\ell,r} = \del_{b_r}\left(\*D^1_{k, j, \ell} \right) = \delta_{r,j}\delta_{\ell,k} + \delta_{r,k}\delta_{j,\ell}$. This would mean that the Hessian of $\Psi$ is very sparse and possibly not negative-definite.
	\end{remark}
	
	\begin{example}[Heat equation with lumped control] \label{ex: 1}
	
	We begin this section by considering a finite difference discretization of the one-dimensional heat equation
	\begin{equation*}
	\begin{dcases}
	y_t(t,x) - y_{xx}(t,x) = b(x) u(t) &\text{ in } (0,T)\times (0,1),\\
	y(t,0) = y(t,1) = 0 &\text{ in } (0,T).\\
	\end{dcases}
	\end{equation*}
	Here $b(x)\in\R$ is a scalar function designating the location wherein the controller actuates with amplitude $u(t)$ in each time $t$.
	By using the classical two-point difference scheme for approximating the second derivative, we obtain the system 
	\begin{equation}
	y_h'(t) - A_{\Delta,h}y_h(t) = b_h u(t) \hspace{1cm} \text{ in } (0,T).
	\end{equation}
	Here $h=\frac{1}{n-1}$ where $n\geqslant2$ represents the number of spatial grid points, with $b_h\in \R^n$ representing the optimization variable, and
	\begin{equation*}
	A_{\Delta,h} := 
	\frac{1}{h^2}\begin{bmatrix} 
           -2& 1& 0& \ldots& 0 \\ 
           1& -2& 1& & \vdots \\
           0& \ddots &\ddots& \ddots& 0 \\
           \vdots& & 1 & -2 & 1 \\
           0& \hdots &0 &1 & -2 
    \end{bmatrix}
	\end{equation*}
	being the standard finite-difference discretization of the Dirichlet Laplacian. 
	
	Let us henceforth address a couple of illustrative cases. We provide illustrations of the results in Figure \ref{fig: ex1.1} and Figure \ref{fig: ex1.2}.
	\smallskip
	
	\noindent \textbf{Case 1):}
	$(n=2)$.
	We shall begin by focusing our attention on the case $n=2$, and thus consider 
	\begin{equation*}
	A_{\Delta} = 
	\begin{bmatrix} 
	-2 & 1\\
	1 & -2
	\end{bmatrix}, \hspace{1cm} b = \begin{bmatrix}b_1\\b_2\end{bmatrix}.
	\end{equation*} 
	In this case, several computations can be done explicitly. Indeed, first note that 
	\begin{align*}
	P(b)P(b)^\top=\begin{bmatrix}
	(2b_1+b_2)^2+b_1^2 & (2b_1+b_2)(b_1+2b_2)+b_1b_2 \\
	(2b_1+b_2)(b_1+2b_2)+b_1b_2 & (b_1+2b_2)^2+b_2^2
	\end{bmatrix},
	\end{align*}
	whence 
	\begin{align*}
	\lambda_1\Big(P(b)P(b)^\top&\Big)= 4b_1b_2+3\big(b_1^2+b_2^2\big)-2\Big(\big(2b_1^2+2b_1b_2+b_2^2\big)\big(b_1^2+2b_1b_2+2b_2^2\big)\Big)^{\frac{1}{2}}.
	\end{align*}
	Making use of Lagrange multipliers and symbolic computation, one can find that the above function has $4$ maximizers.
	Numerically, we find the following $4$ maximizers: 
	\begin{align} \label{eq: ex1.1.max}
	b^* = \begin{bmatrix} b_1^*\\ b_2^*\end{bmatrix} \in \left\{\begin{bmatrix} -0.257983\\0.96614944\end{bmatrix}, \begin{bmatrix} 0.257983\\-0.96614944\end{bmatrix}, \begin{bmatrix} 0.96614944\\-0.257983\end{bmatrix}, \begin{bmatrix} -0.96614944\\0.257983\end{bmatrix} \right\}.
	\end{align}
	We depict these maximizers on $\mathbb{S}^1$ in Figure \ref{fig: ex1.1}.
	Interestingly enough, we see that
	\begin{align*}
	\begin{bmatrix}-1&0\\0&-1\end{bmatrix}\begin{bmatrix} 0.96614944\\-0.257983\end{bmatrix} &= \begin{bmatrix} -0.96614944\\0.257983\end{bmatrix}\\
	\begin{bmatrix}0&1\\1&0\end{bmatrix}\begin{bmatrix} 0.96614944\\-0.257983\end{bmatrix} &= \begin{bmatrix} -0.257983\\0.96614944\end{bmatrix}\\
	\begin{bmatrix}0&-1\\-1&0\end{bmatrix}\begin{bmatrix} 0.96614944\\-0.257983\end{bmatrix} &= \begin{bmatrix} 0.257983\\-0.96614944\end{bmatrix},
	\end{align*}
	whence one may generate all the maximizers from $[0.96614944, -0.257983]^\top$ and applying the orthogonal (rotation) matrices appearing in the identities just above, all of which commute with $A_\Delta$. This may also be seen in Figure \ref{fig: ex1.1}.
		
	\begin{figure}
	\includegraphics[scale=0.42]{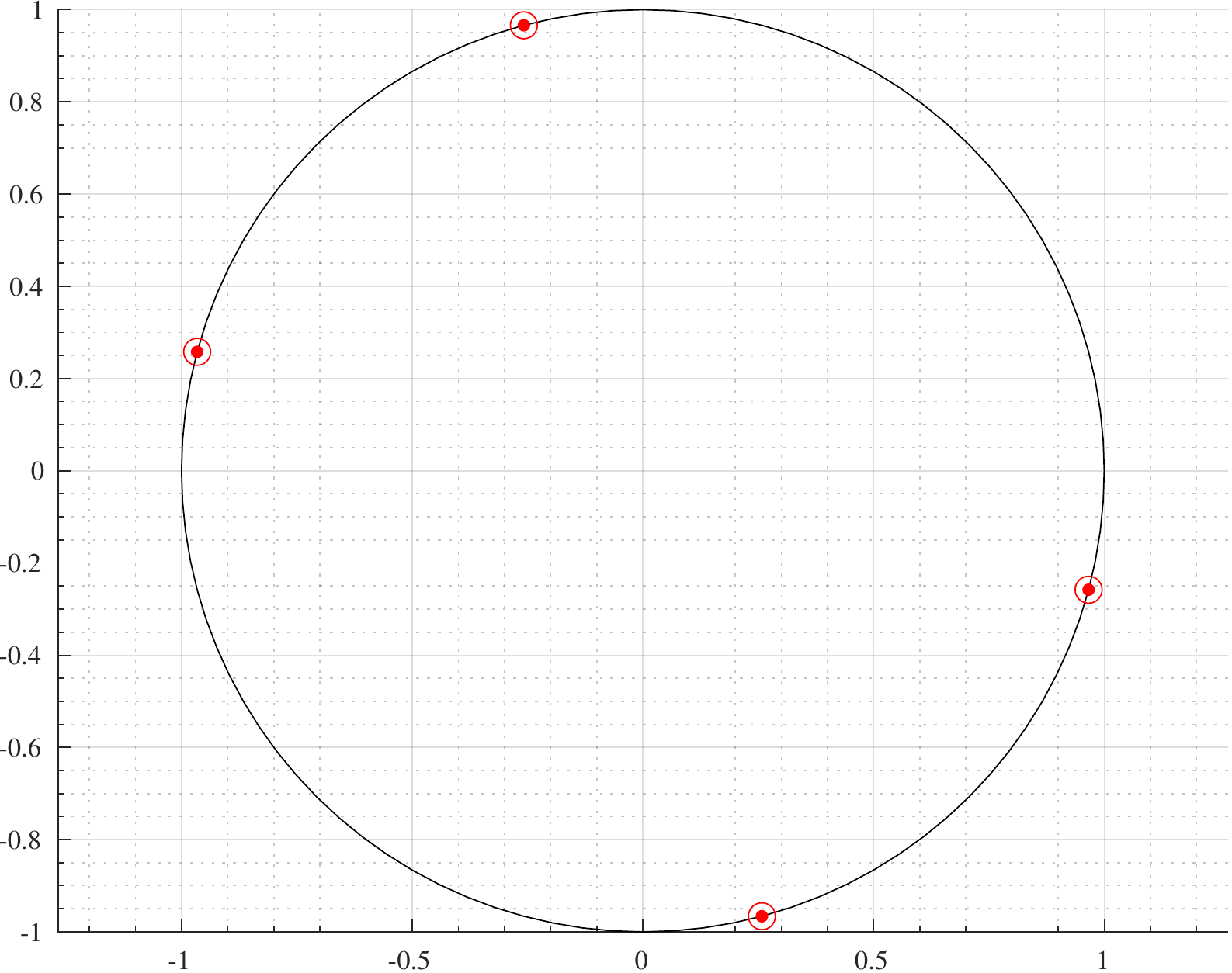}
	\hspace{0.1cm}
	\includegraphics[scale=0.42]{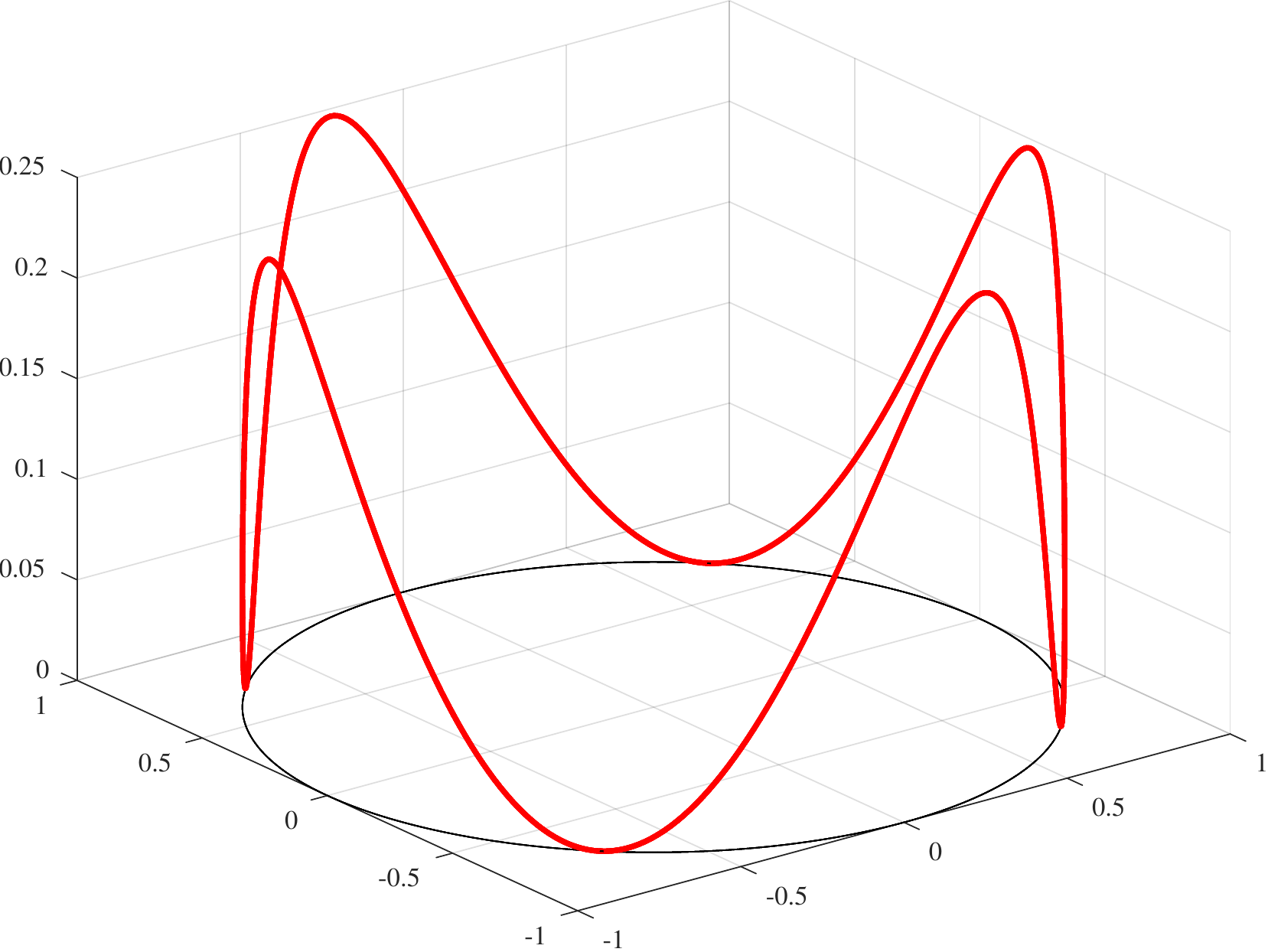}
	\caption{\textbf{Example \ref{ex: 1}} ($n=2$). \emph{Left}: The $4$ maximizers on $\mathbb{S}^1$ found by the \texttt{IPOPT} algorithm, as indicated in \eqref{eq: ex1.1.max}. \emph{Right}: the graph of the function $\mathbb{S}^1\ni b\mapsto\lambda_1(P(b)P(b)^\top)$, wherein we see 1). the maximum equal to $0.24913$ attained at the computed maximizers located on the left plot; 2). the zeros are attained at points which do not satisfy the Kalman rank condition, which are precisely the $4$ points with $|b_1|=|b_2|=\frac{\sqrt{2}}{2}$; 3). the rotational symmetry of the cost functional.}
	\label{fig: ex1.1}
	\end{figure}
	
	\smallskip
	\noindent \textbf{Case 2):} $(n=3)$. We also provide the numerical results in the case $n=3$, and depict the functional to be maximized in Figure \ref{fig: ex1.2}.  We numerically find the following $8$ maximizers: 
	\begin{align} \label{eq: ex1.2.max}
	b^* \in \Bigg\{&\begin{bmatrix} -0.7633\\ 0.6325\\0.1311\end{bmatrix}, 
\begin{bmatrix} -0.1311\\ 0.6325\\ -0.7633\end{bmatrix},
\begin{bmatrix} -0.1311\\ -0.6325\\ 0.7633\end{bmatrix},\begin{bmatrix} 0.7633\\ -0.6325\\ -0.1311\end{bmatrix},\\
&\begin{bmatrix} -1.346*10^{-7}\\ 0.44707\\ -0.8944\end{bmatrix}, \begin{bmatrix} 4.975*10^{-7}\\ -4.44707\\ 0.8944\end{bmatrix}, \begin{bmatrix} -9.089*10^{-8}\\ 0.44707\\ -0.8944\end{bmatrix}, 
\begin{bmatrix} -4.8519*10^{-8}\\ 0.44707\\-0.8944\end{bmatrix} \Bigg\} \nonumber.
	\end{align}
	We again note a similar rotational symmetry among the obtained maximizers. The latter can be visualized as the peaks in brightly colored patches in Figure \ref{fig: ex1.2}. We do not conjecture that these maximizers are the sole ones that the functional possesses, as the yellow patches appearing in Figure \ref{fig: ex1.2} could contain multiple peaks.
	
	\begin{SCfigure}
	\includegraphics[scale=0.5]{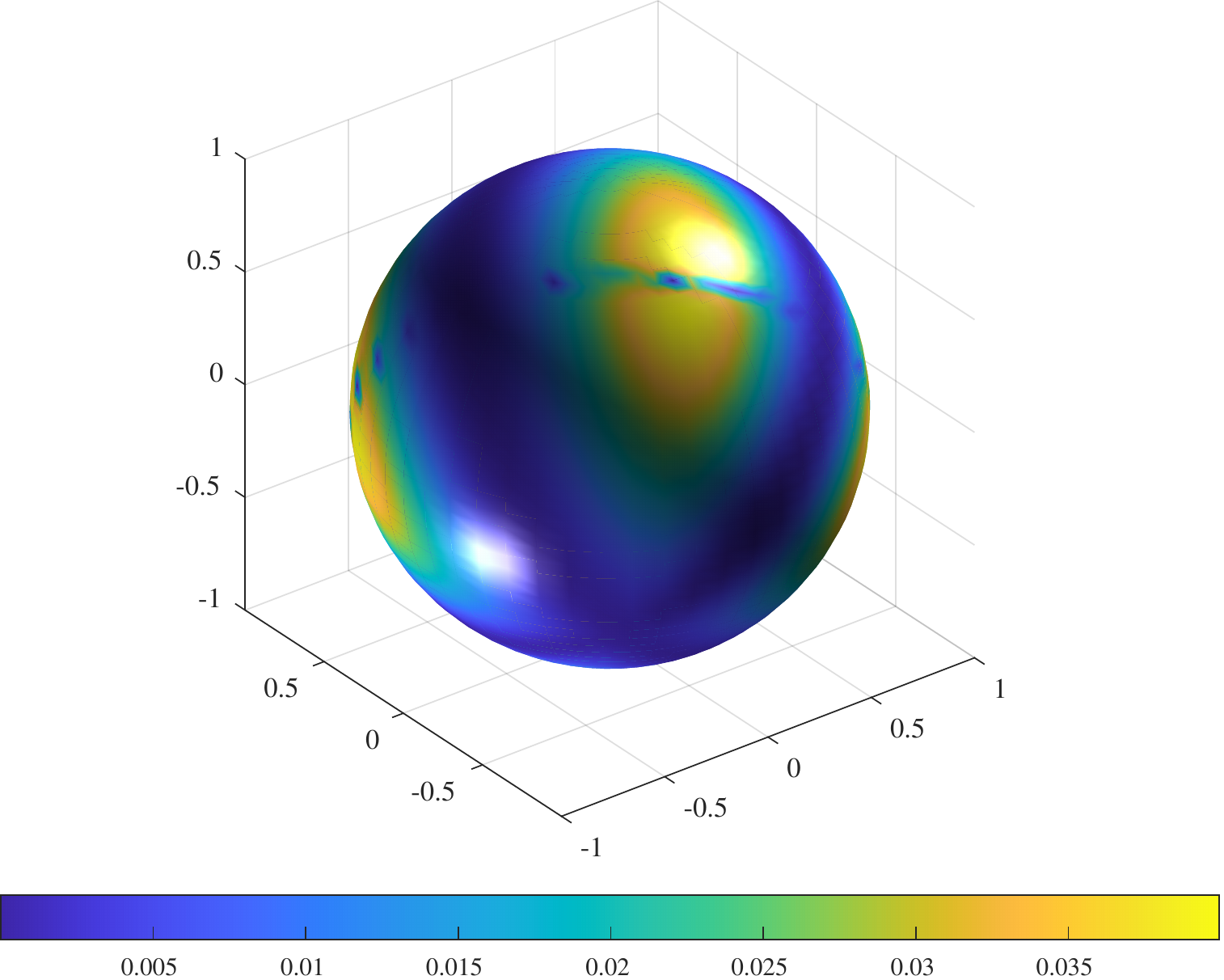}	
	\caption{\textbf{Example \ref{ex: 1}} ($n=3$). The functional $b\mapsto \lambda_1(P(b)P(b)^\top)$ on $\mathbb{S}^2$; the opposite side of the sphere manifests the same pattern. We dispose of $8$ maximizers at which the maximum value equal to $\sim0.0399$ is attained. 
	Rotational symmetry is also apparent.}
	  \label{fig: ex1.2}
	  \end{SCfigure}
	\end{example}
	
	\begin{example}[Wave equation with lumped control] \label{ex: 3}
	
	We now consider a finite-difference discretization of the one-dimensional wave equation with lumped control: 
	\begin{equation*}
	\begin{dcases}
	z_{tt}(t,x) - z_{xx}(t,x) = b(x)u(t) &\text{ in } (0,T)\times(0,1),\\
	z(t,0) = z(t,1) = 0 &\text{ in } (0,T).
	\end{dcases}	
	\end{equation*}
	By setting $y:=[z, z_t]^\top$, we rewrite the equation in the above system in the canonical first-order form as
	\begin{equation*}
	y_t(t,x) - \begin{bmatrix}0&\mathrm{Id}\\ \del_x^2&0\end{bmatrix}y(t,x) = \begin{bmatrix}0\\b(x)\end{bmatrix} u(t) \, \text{ in } (0,T)\times(0,1).
	\end{equation*}
	When the Dirichlet Laplacian is discretized as in the previous examples, we find ourselves with a linear control system in $\R^{2n}$, with system dynamics
	\begin{equation*}
	A_{\square, h}:= \begin{bmatrix} 0&\mathrm{Id}_n\\
						      A_{\Delta,h}&0
	\end{bmatrix}
	\end{equation*}
	with $A_{\Delta, h}$ as in Example \ref{ex: 1}. We depict the shape of the functional $b\mapsto \lambda_1(P(b)P(b)^\top)$ in Figure \ref{fig: ex3.1} $(n=2)$ and Figure \ref{fig: ex3.2} $(n=3)$. 
	We in fact see that the functional is identical to that of the heat case, thus the found maximizers are as well. This is due to the following result.
	
	\begin{proposition} \label{prop: wave=heat} Let $P_{\square}(b)\in\mathrm{GL}_{2n}(\R)$ denote the change-of-basis matrix for $A_{\square,h}\in\mathcal{M}_{2n\times 2n}$, and $P_{\Delta}(b)\in\mathrm{GL}_n(\R)$ that for $A_{\Delta,h}\in\mathcal{M}_{n\times n}(\R)$. Then 
	\begin{equation}
	P_\square(b)P_\square(b)^\top = \begin{bmatrix} P_{\Delta}(b)P_{\Delta}(b)^\top & 0\\ 0 & P_{\Delta}(b)P_{\Delta}(b)^\top\end{bmatrix}.
	\end{equation}
	Consequently, $\lambda_1\left(P_\square(b)P_\square(b)^\top\right) = \lambda_1\left(P_{\Delta}(b)P_{\Delta}(b)^\top\right)$.
	\end{proposition}
	
	\begin{figure}	
	\includegraphics[scale=0.35]{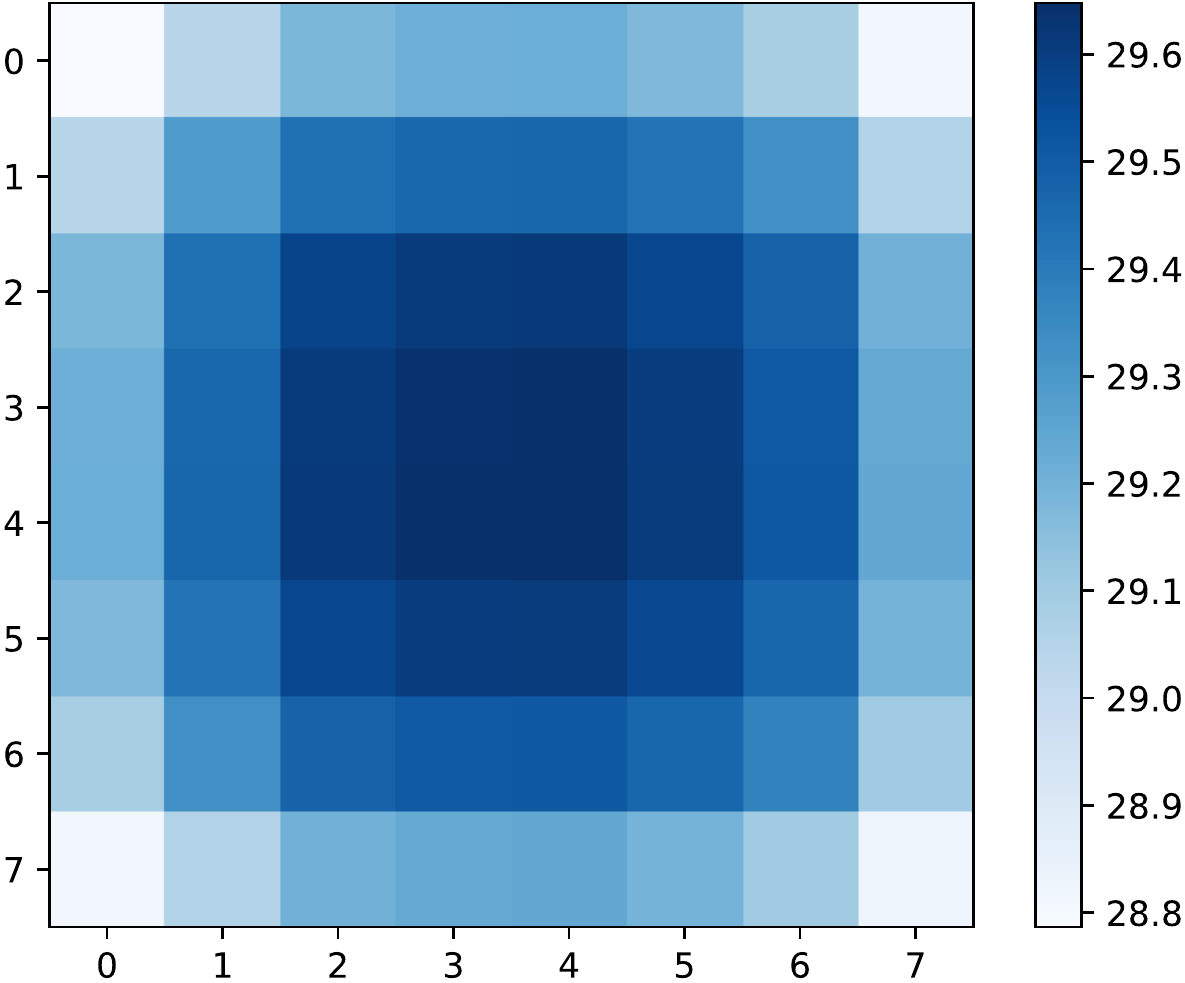}
	\hspace{0.1cm}
	\includegraphics[scale=0.35]{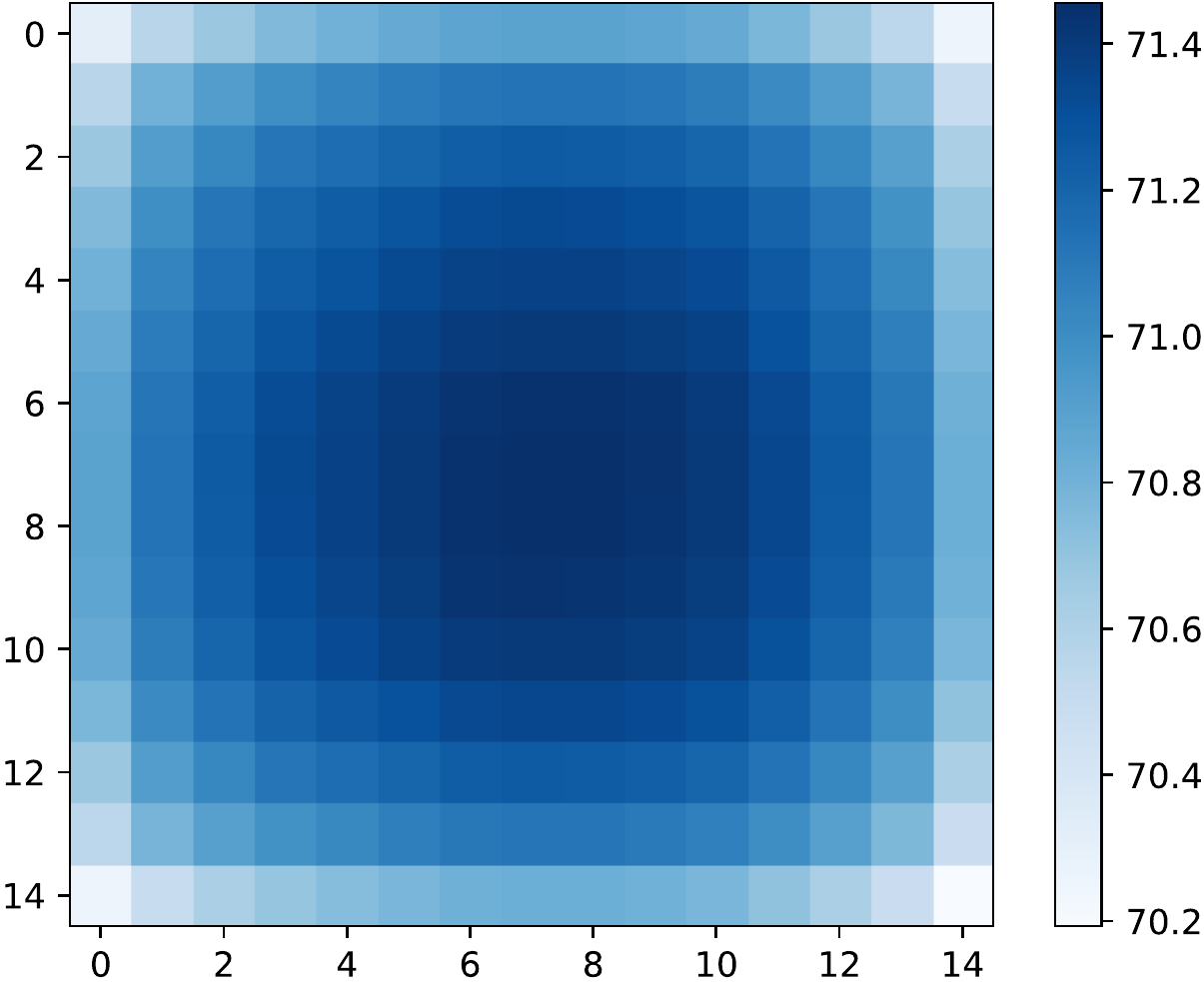}
	\hspace{0.1cm}
	\includegraphics[scale=0.35]{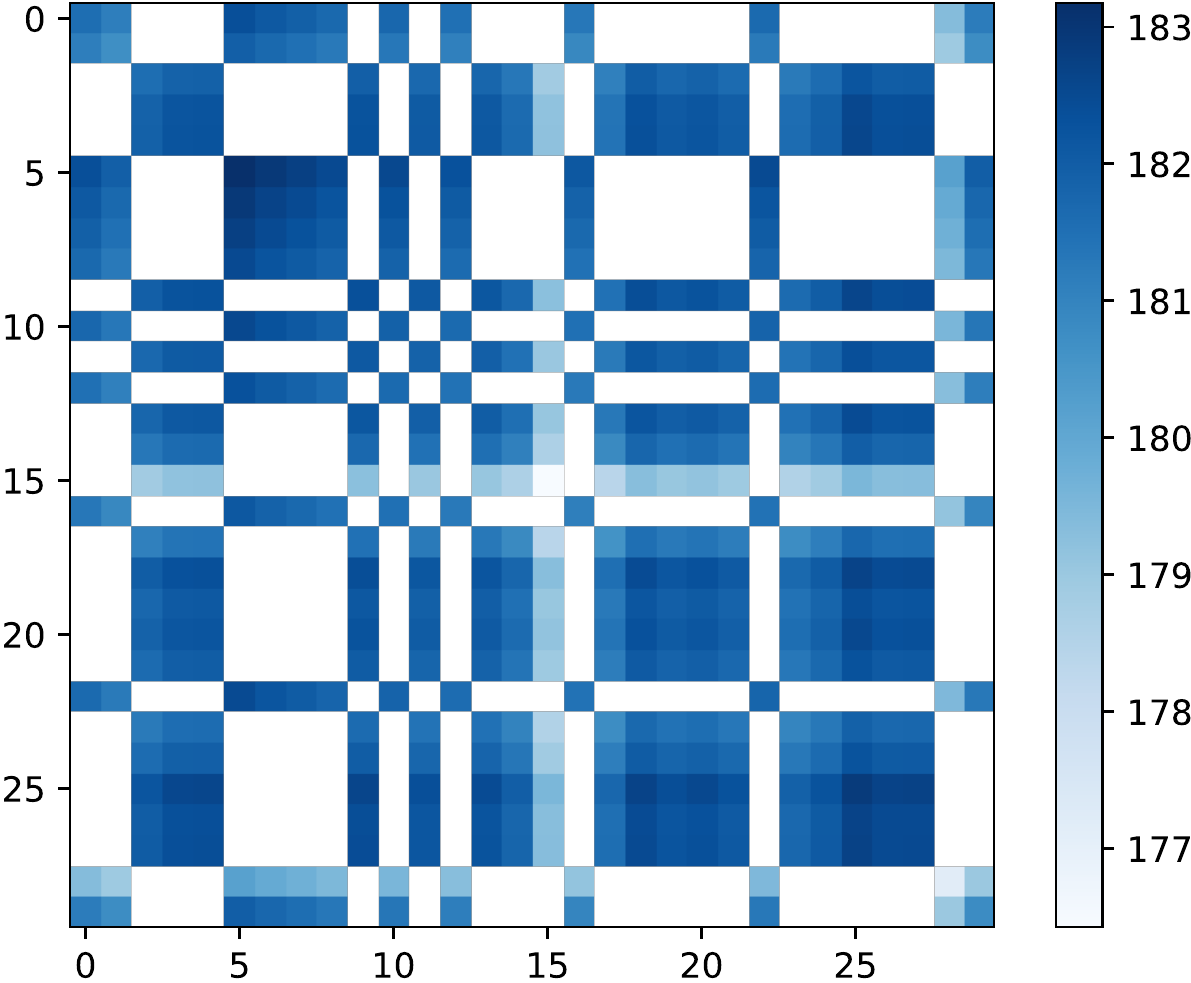}
	\vspace{0.25cm}
	
	\includegraphics[scale=0.35]{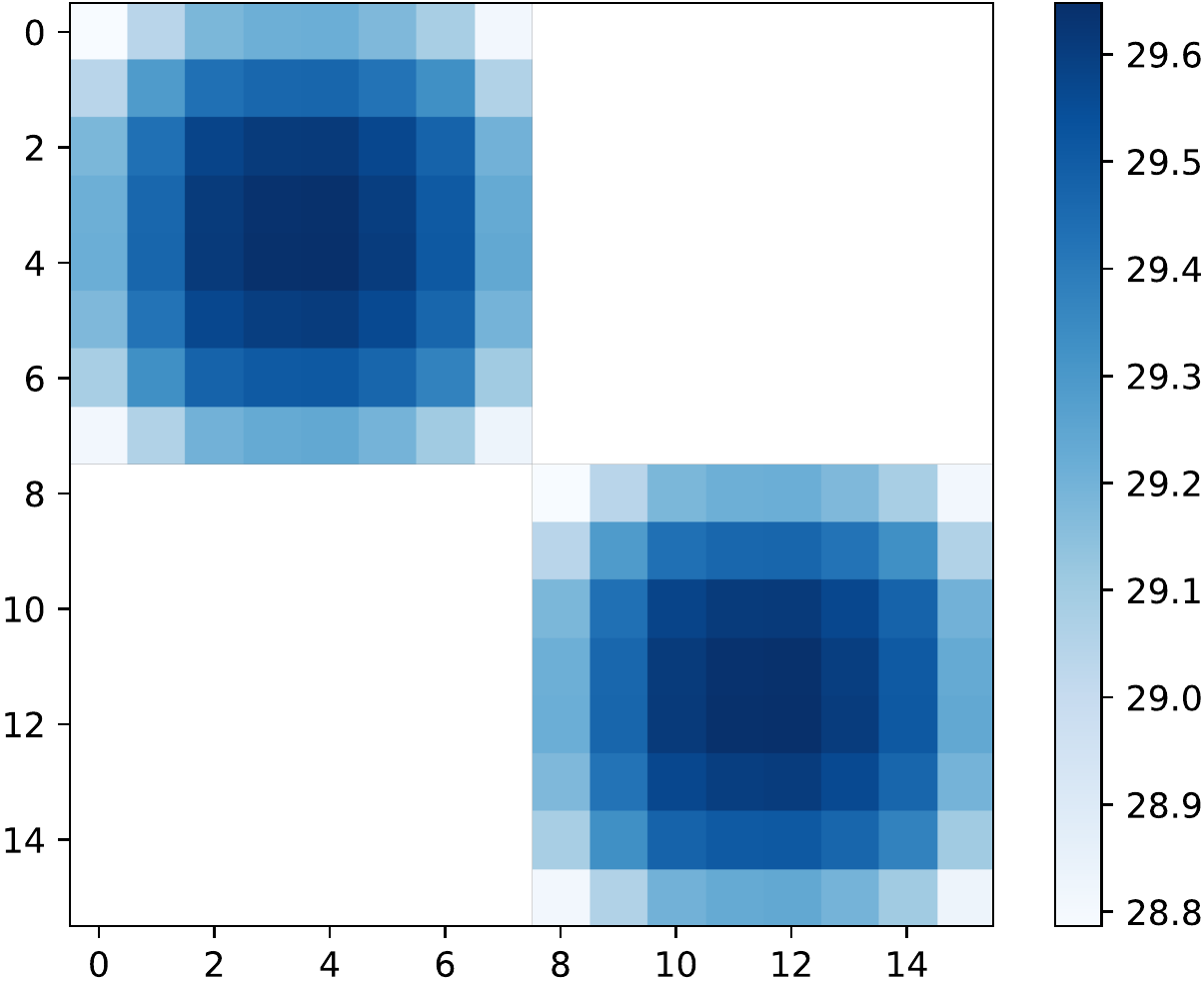}
	\hspace{0.1cm}
	\includegraphics[scale=0.35]{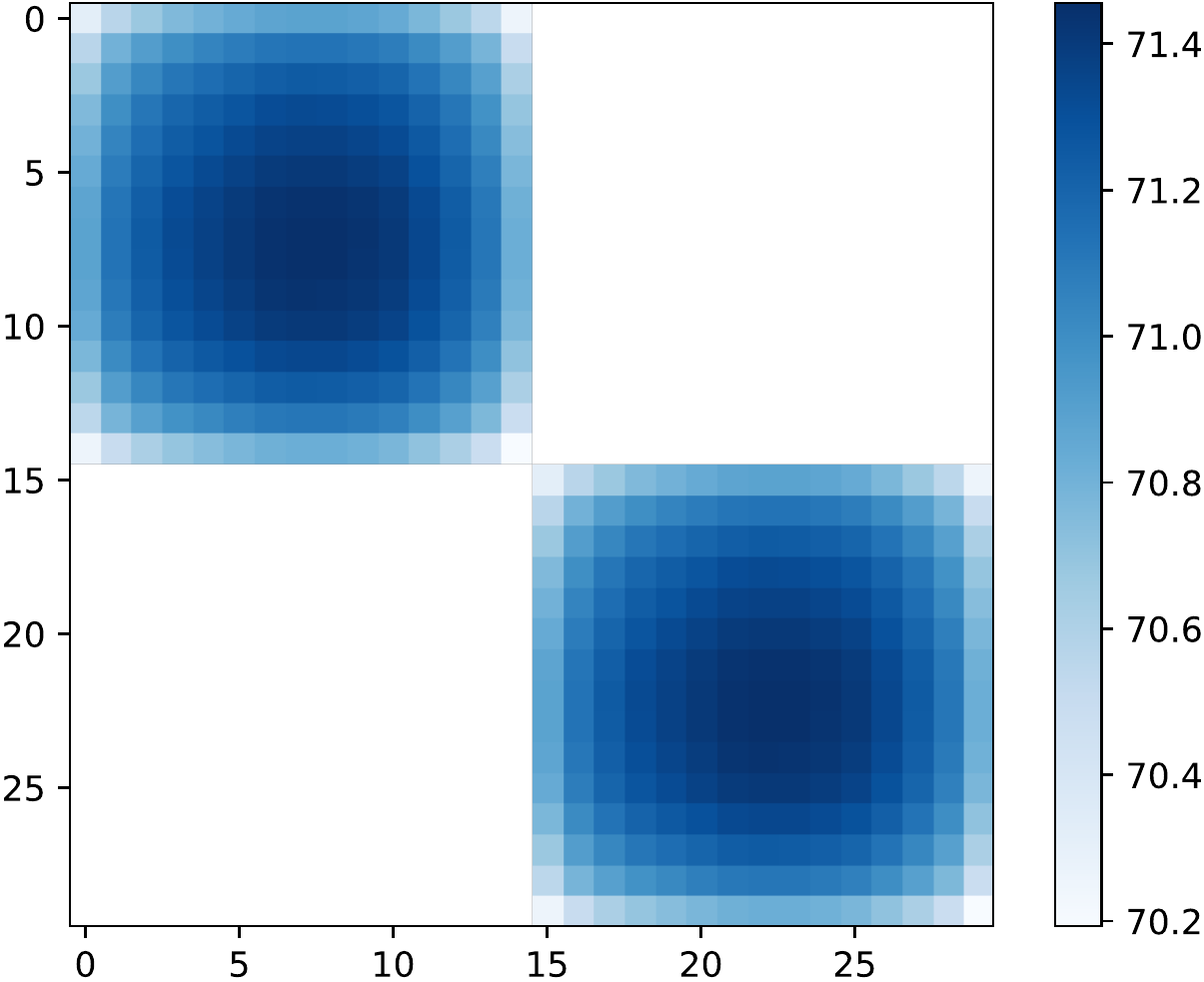}
	\hspace{0.1cm}
	\includegraphics[scale=0.35]{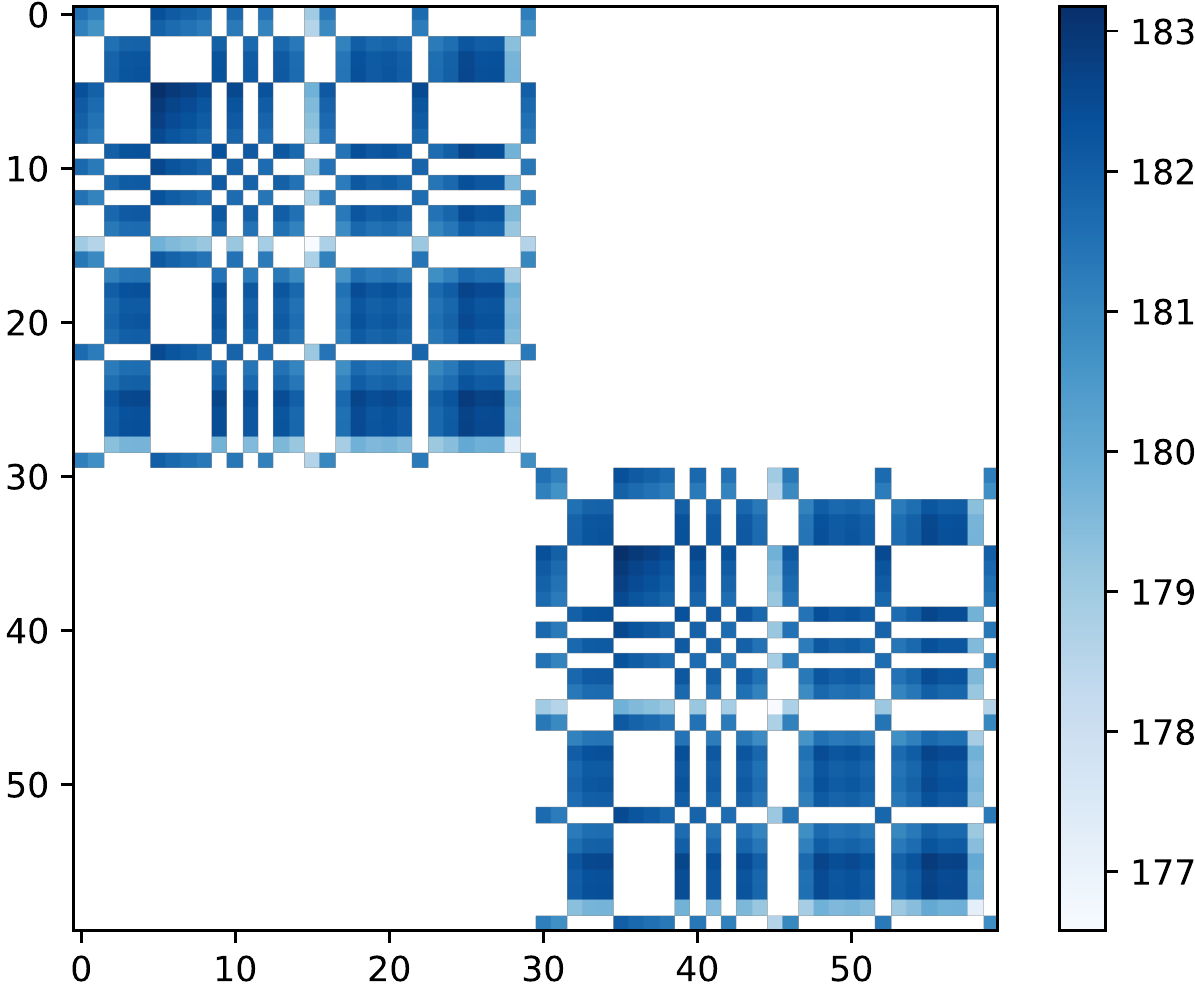}
	\caption{Graphical depiction of Proposition \ref{prop: wave=heat}: we display $P_{\Delta}(b)P_{\Delta}(b)^\top$ (top) and $P_{\square}(b)P_{\square}(b)^\top$ (bottom) for $n\in\{8, 15, 30\}$, with $b$ picked at random per each selected $n$. More precisely, we display the $\log_{10}$ of these matrices to enhance visibility. An interesting pattern starts to appear for $n\geqslant26$ as seen on the rightmost figures, likely due to dimensionality.}
	\end{figure}
	
	\begin{proof}[Proof of Proposition \ref{prop: wave=heat}]
	We begin by recalling that (we drop the indexes $h$)
	\begin{equation*}
	P_\square(b)P_\square(b)^\top = \sum_{k=1}^{2n} p_k(A_{\square}) \begin{bmatrix} 0 & 0\\
	0 & bb^\top\end{bmatrix} p_k(A_{\square})^\top,
	\end{equation*}
	with 
	\begin{equation*}
	p_k(A_{\square}) := \begin{dcases} A_\square^{2n-k} + \sum_{j=1}^{2n-k} a^\square_j A_\square^{2n-k-j}  &k\leqslant 2n-1, \\
	\mathrm{Id}_{2n} &k=2n.
	\end{dcases}
	\end{equation*}
	We distinguish two cases.
	\smallskip
	
	\noindent
	\textbf{Case 1): $k$ is even.} One can easily show by induction that
	\begin{equation}
	A_\square^k = \begin{bmatrix} A_\Delta^{\frac{k}{2}} & 0 \\
	0 & A_{\Delta}^{\frac{k}{2}}\end{bmatrix},
	\end{equation}
	and, moreover, $a_j^\square=0$ for $j$ odd and $a_{2j}^\square=a_{j}^\Delta$ for $j$ even. Hence,
	\begin{align*}
	p_k(A_\square) &= \begin{bmatrix} A_\Delta^{\frac{2n-k}{2}} & 0 \\
	0 & A_{\Delta}^{\frac{2n-k}{2}}\end{bmatrix} + \sum_{j=2}^{2n-k} a_{\frac{j}{2}}^\Delta \begin{bmatrix} A_\Delta^{\frac{2n-k-j}{2}} & 0 \\
	0 & A_{\Delta}^{\frac{2n-k-j}{2}}\end{bmatrix}.
	\end{align*}
	Setting $k=2\kappa$ and $j=2r$, we see that
	\begin{align*}
	p_{2\kappa}(A_\square) &= \begin{bmatrix} A_\Delta^{n-\kappa} & 0 \\
	0 & A_{\Delta}^{n-\kappa}\end{bmatrix} + \sum_{r=1}^{n-\kappa} a_{r}^\Delta \begin{bmatrix} A_\Delta^{n-\kappa-r} & 0 \\
	0 & A_{\Delta}^{n-\kappa-r}\end{bmatrix} = \begin{bmatrix} p_{\kappa}(A_\Delta) & 0 \\ 
	0 & p_{\kappa}(A_\Delta) \end{bmatrix}.
	\end{align*}
	Consequently, for $k=2\kappa$, $\kappa\geqslant1$, 
	\begin{align} \label{eq: 4.7}
	&p_{2\kappa}(A_\square)\begin{bmatrix} 0 & 0\\
	0 & bb^\top\end{bmatrix} p_{2\kappa}(A_{\square})^\top \nonumber \\
	&\quad= \begin{bmatrix} 0 & 0 \\ 0 & p_{\kappa}(A_\Delta) bb^\top p_{\kappa}(A_\Delta)^\top\end{bmatrix}.
	\end{align}
	
	\smallskip
	
	\noindent
	\textbf{Case 2): $k$ is odd.} One can, once again, easily show by induction that
	\begin{equation*}
	A_\square^k = \begin{bmatrix} 0 & A_\Delta^{\frac{k-1}{2}} \\
	A_\Delta^{\frac{k+1}{2}} & 0 \end{bmatrix}.
	\end{equation*}
	Hence, 
	\begin{align*}
	p_k(A_\square) &= \begin{bmatrix} 0 & A_\Delta^{\frac{2n-k-1}{2}} \\
	A_\Delta^{\frac{2n-k+1}{2}} & 0 \end{bmatrix}+ \sum_{j=2}^{2n-k-1} a_{\frac{j}{2}}^\Delta \begin{bmatrix} 0 & A_\Delta^{\frac{2n-k-j-1}{2}} \\
	A_\Delta^{\frac{2n-k-j+1}{2}} & 0 \end{bmatrix}.
	\end{align*}
	By setting $k=2\kappa-1$ with $\kappa\geqslant1$, and $j=2r$, we find 
	\begin{align*}
	p_{2\kappa-1}(A_\square) &=  \begin{bmatrix} 0 & A_\Delta^{n-\kappa} \\
	A_\Delta^{n-\kappa+1} & 0 \end{bmatrix} + \sum_{r=1}^{n-\kappa} a_r^\Delta \begin{bmatrix} 0 & A_\Delta^{n-\kappa-r} \\
	A_\Delta^{n-\kappa+r+1} & 0 \end{bmatrix}.
	\end{align*}
	It then follows that for $\kappa\geqslant1$, 
	\begin{align} \label{eq: 4.8}
	p_{2\kappa-1}(A_\square)\begin{bmatrix} 0 & 0\\
	0 & bb^\top\end{bmatrix} p_{2\kappa-1}(A_{\square})^\top =  \begin{bmatrix}p_{\kappa}(A_\Delta) bb^\top p_{\kappa}(A_\Delta)^\top & 0 \\ 0 & 0\end{bmatrix}.
	\end{align}
	Combining \eqref{eq: 4.7} and \eqref{eq: 4.8}, we may conclude.
	\end{proof}

	\begin{figure}
	\includegraphics[scale=0.42]{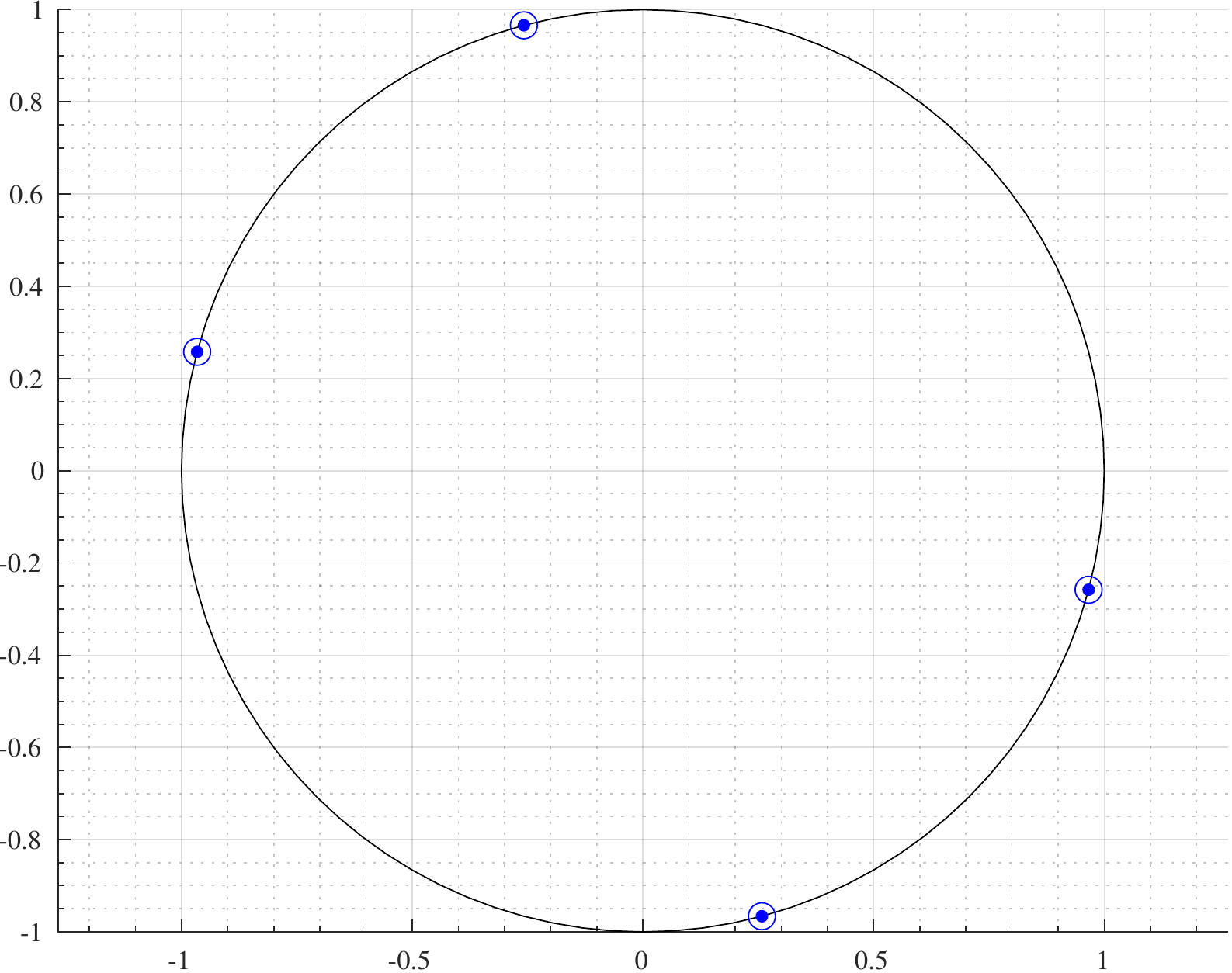}
	\hspace{0.1cm}
	\includegraphics[scale=0.42]{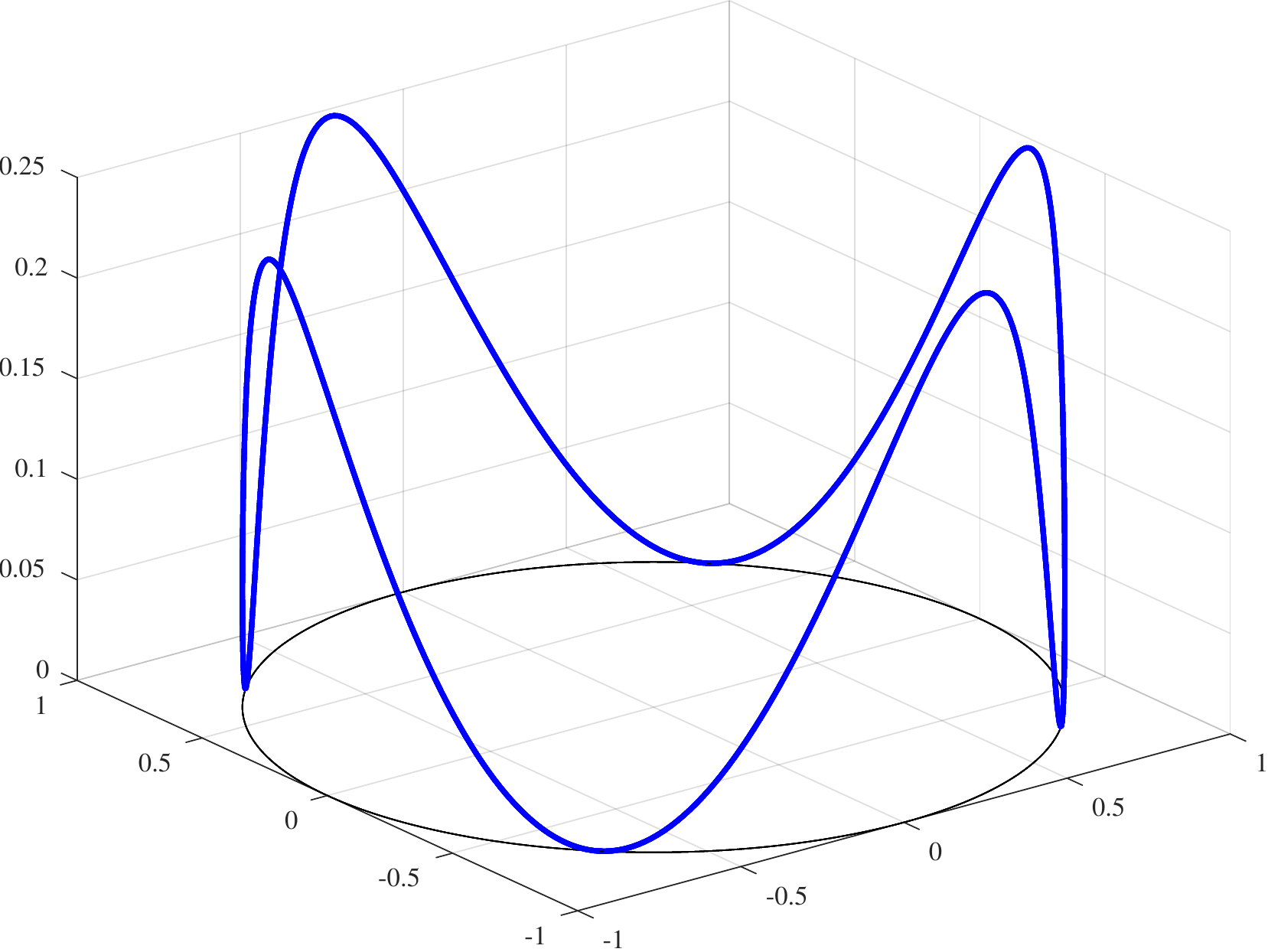}
	\caption{\textbf{Example \ref{ex: 3}} ($n=2$).  The maximizers and the functional are identical to the heat system in Example \ref{ex: 1}.}
	\label{fig: ex3.1}
	\end{figure}
	
	\begin{SCfigure}
	\includegraphics[scale=0.5]{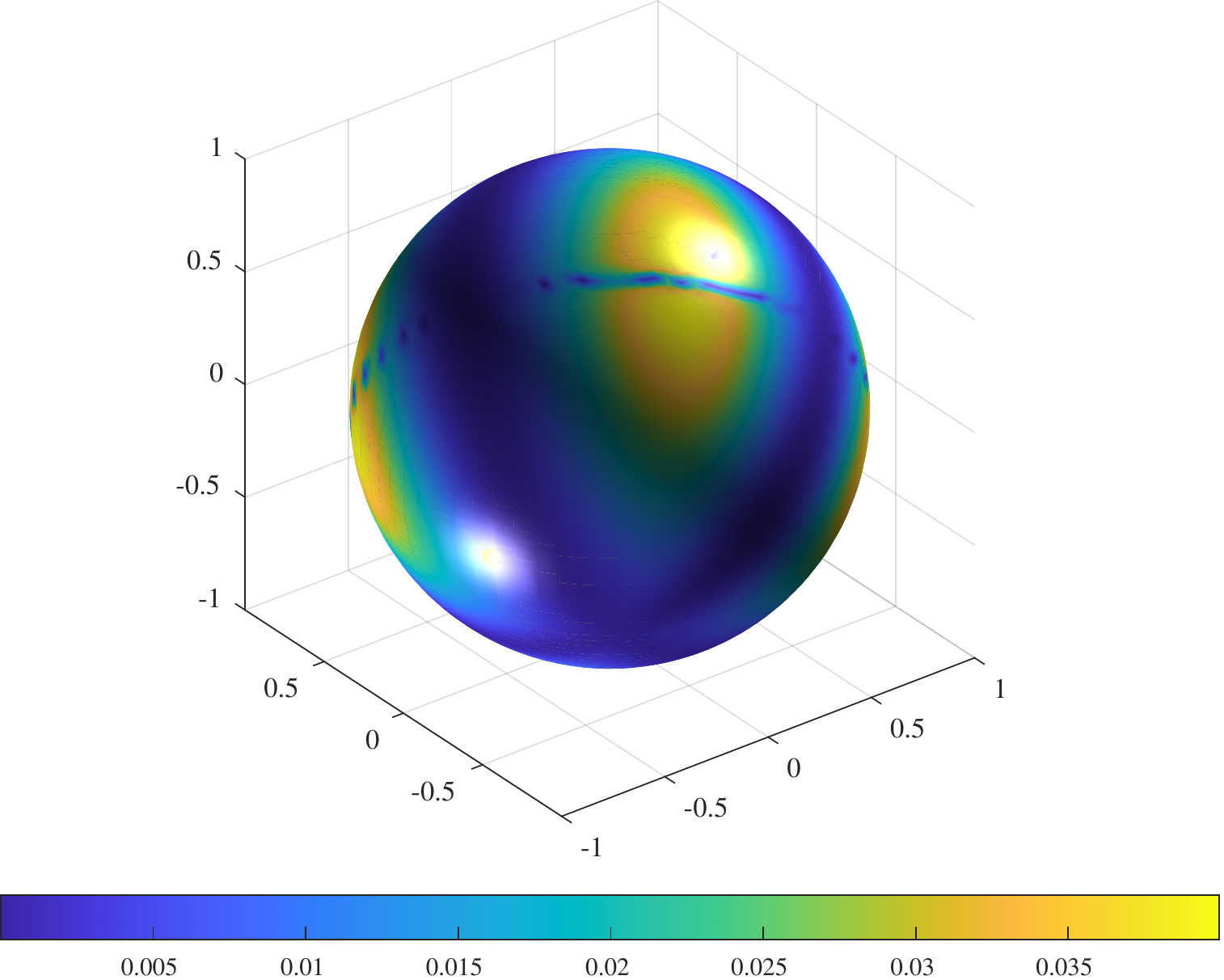}	
	\caption{\textbf{Example \ref{ex: 3}} ($n=3$). The functional $b\mapsto \lambda_1(P(b)P(b)^\top)$ on $\mathbb{S}^2$ (and thus the maximizers) are the same as for the heat system in Example \ref{ex: 1}.}
	  \label{fig: ex3.2}
	\end{SCfigure}
	
	\end{example}

	\begin{example}[Advection-diffusion equation with lumped control] \label{ex: 2}
	
	We now consider a system which is non-diagonalizable, hence existing methods based on randomization are not applicable.
 Namely, we consider the finite difference discretization of the one-dimensional advection-diffusion equation
	\begin{equation*}
	\begin{dcases}
	y_t(t,x) - y_{xx}(t,x) + y_x(t,x) = b(x) u(t) &(0,T)\times (0,1),\\
	y(t,0) = y(t,1) = 0 &(0,T),\\
	\end{dcases}
	\end{equation*}
	as well as
	\begin{equation*}
	\begin{dcases}
	y_t(t,x) - y_{xx}(t,x) - y_x(t,x) = b(x) u(t) &(0,T)\times (0,1),\\
	y(t,0) = y(t,1) = 0 &(0,T).\\
	\end{dcases}
	\end{equation*}
	Using a finite difference approximation as for Example \ref{ex: 1} and in particular a centered difference scheme for the advection term, we obtain a couple of finite-dimensional control systems with system dynamics of the form
	\begin{align*}
	A_{\pm\del_x} &:= \frac{1}{h^2}
	\begin{bmatrix} 
           -2& 1& 0& \ldots& 0 \\ 
           1& -2& 1& & \vdots \\
           0& \ddots &\ddots& \ddots& 0 \\
           \vdots& & 1 & -2 & 1 \\
           0& \hdots &0 &1 & -2 
    \end{bmatrix}  + \frac{1}{2h} \begin{bmatrix} 0& \pm1& 0& \ldots& 0 \\ 
						     \mp1& 0& \pm1& & \vdots \\
						     \vdots&  &\ddots& \ddots& 0 \\
						     0& \ldots& \mp1 & 0 & \pm1 \\
						     0& \hdots &\hdots &\mp1 & 0  
				\end{bmatrix}.
	\end{align*}
	We provide illustrations of the results in Figure \ref{fig: ex2.1} $(n=2)$ and Figure \ref{fig: ex2.2} $(n=3)$.
	
	In the case $n=2$, the (approximate) maximal value of $0.32236$ of the functional (same for both $A_{\del_x}$ and $A_{-\del_x}$) is attained at the points 
	\begin{align} \label{eq: max.ex.2.1}
	b^*_{\del_x} &\in \left\{\begin{bmatrix}-0.9548099\\0.296895\end{bmatrix}, \begin{bmatrix}0.9548099\\-0.296895\end{bmatrix}\right\}, \nonumber\\
	b^*_{-\del_x} &\in \left\{\begin{bmatrix}-0.296895\\0.9548099\end{bmatrix}, \begin{bmatrix}0.296895\\-0.9548099\end{bmatrix}\right\}.
	\end{align}
	Note that the maximizers $b^*_{\del_x}$ and $b^*_{-\del_x}$ are themselves an axial symmetry of one another. 
	
	Similarly, for $n=3$, we find
	\begin{align} \label{eq: max.ex.2.2}
	b^*_{\del_x} \in \left\{\begin{bmatrix}-0.8716\\0.4901\\-9.34*10^{-9}\end{bmatrix}, \begin{bmatrix}-0.8716\\0.4901\\1.246*10^{-6}\end{bmatrix}\begin{bmatrix}0.8716\\-0.4901\\-7.297*10^{-8}\end{bmatrix}, \begin{bmatrix} 0.8716\\-0.4901\\1.541*10^{-7}\end{bmatrix}\right\},
	\end{align}
	as well as 
	\begin{align} \label{eq: max.ex.2.2.1}
	b^*_{-\del_x} \in \left\{\begin{bmatrix}-9.229*10^{-8}\\0.4901\\ -0.8716\end{bmatrix}, \begin{bmatrix}-3.581*10^{-8}\\0.4901\\ -0.8716\end{bmatrix}\begin{bmatrix}-2.223*10^{-7}\\-0.4901\\0.8716\end{bmatrix}, \begin{bmatrix}1.787*10^{-7}\\-0.4901\\0.8716\end{bmatrix}\right\}.
	\end{align}
	
	\begin{figure}
	\includegraphics[scale=0.42]{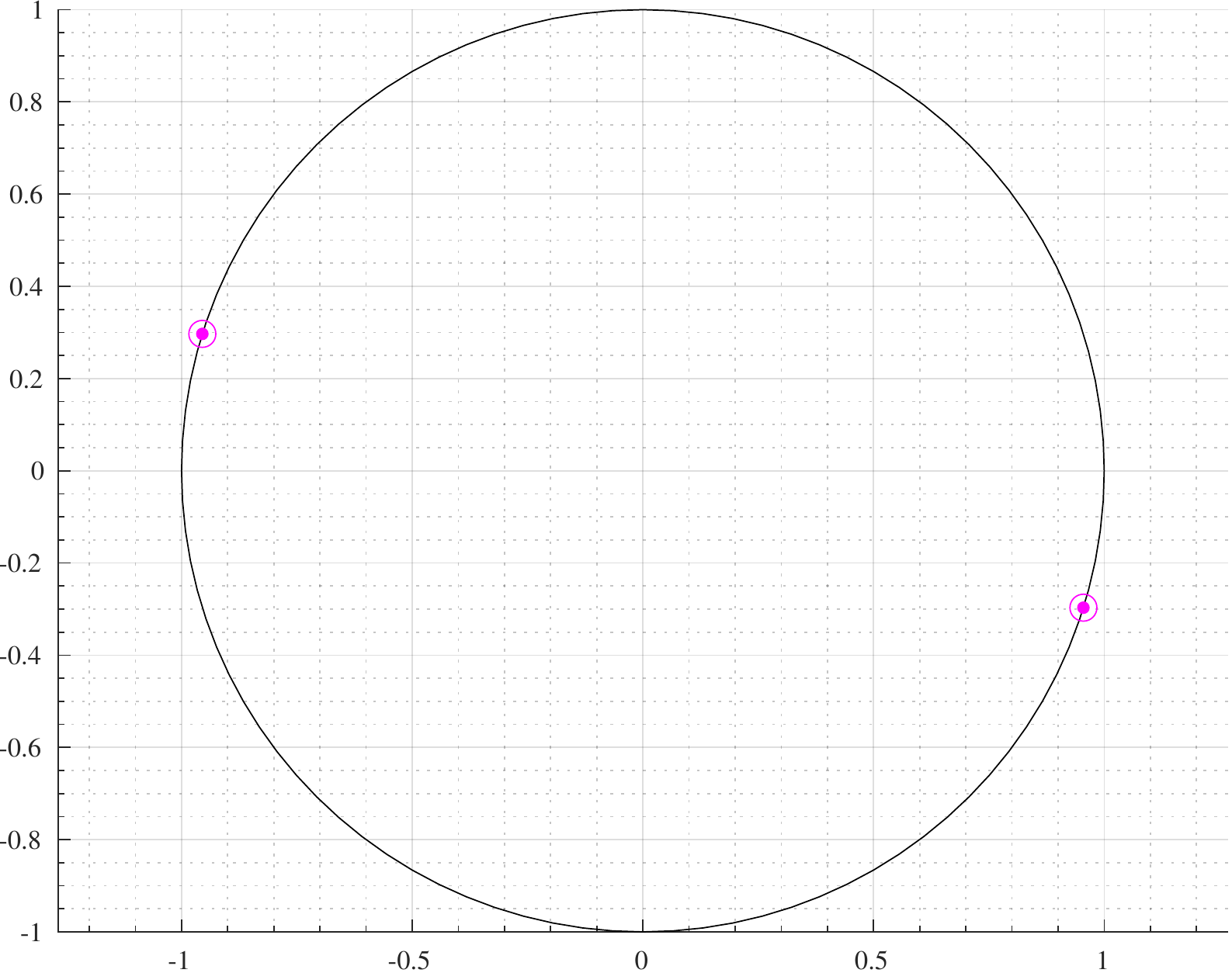}
	\hspace{0.1cm}
	\includegraphics[scale=0.42]{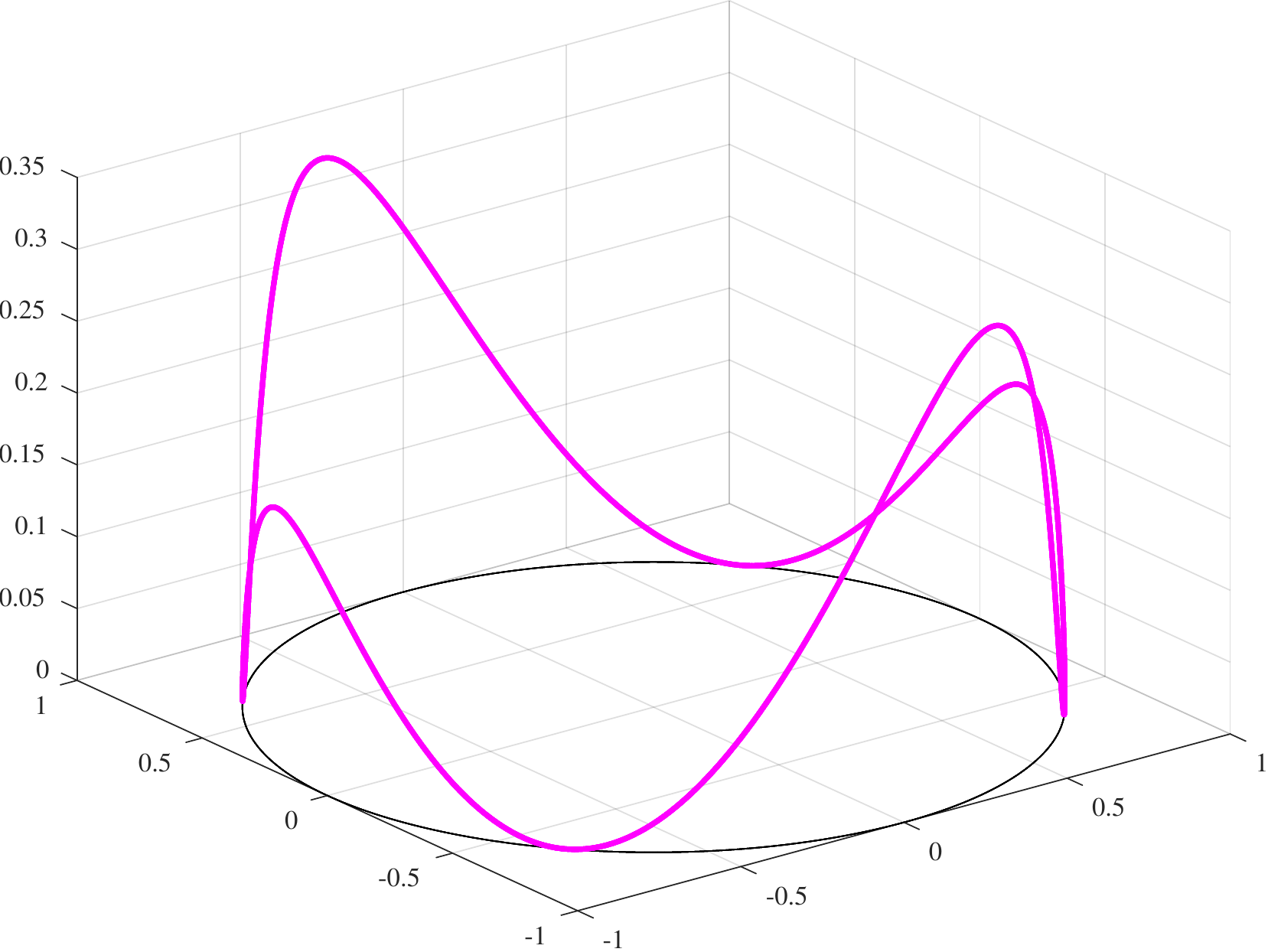}
	\vspace{0.1cm}
	
	\includegraphics[scale=0.42]{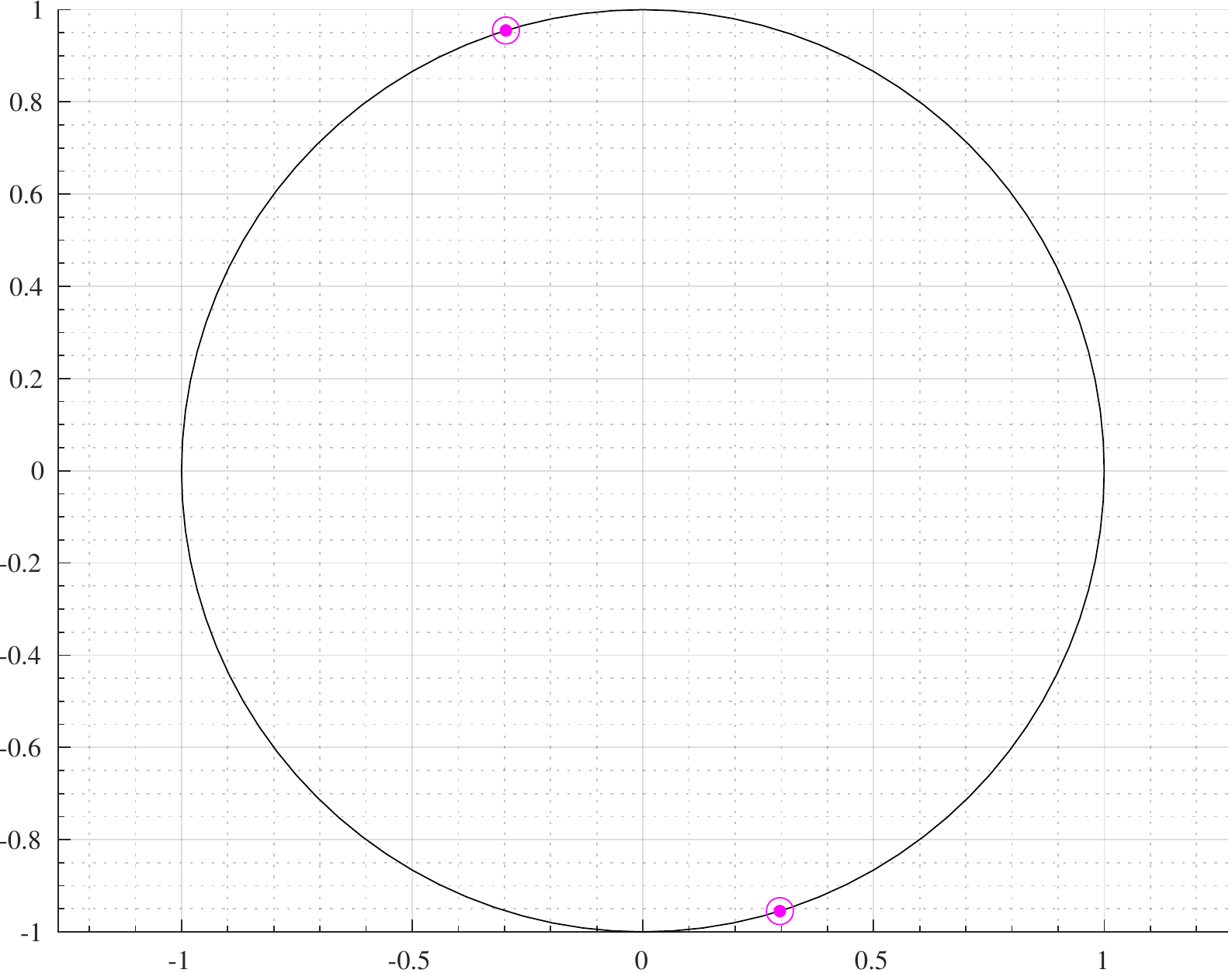}
	\hspace{0.1cm}
	\includegraphics[scale=0.42]{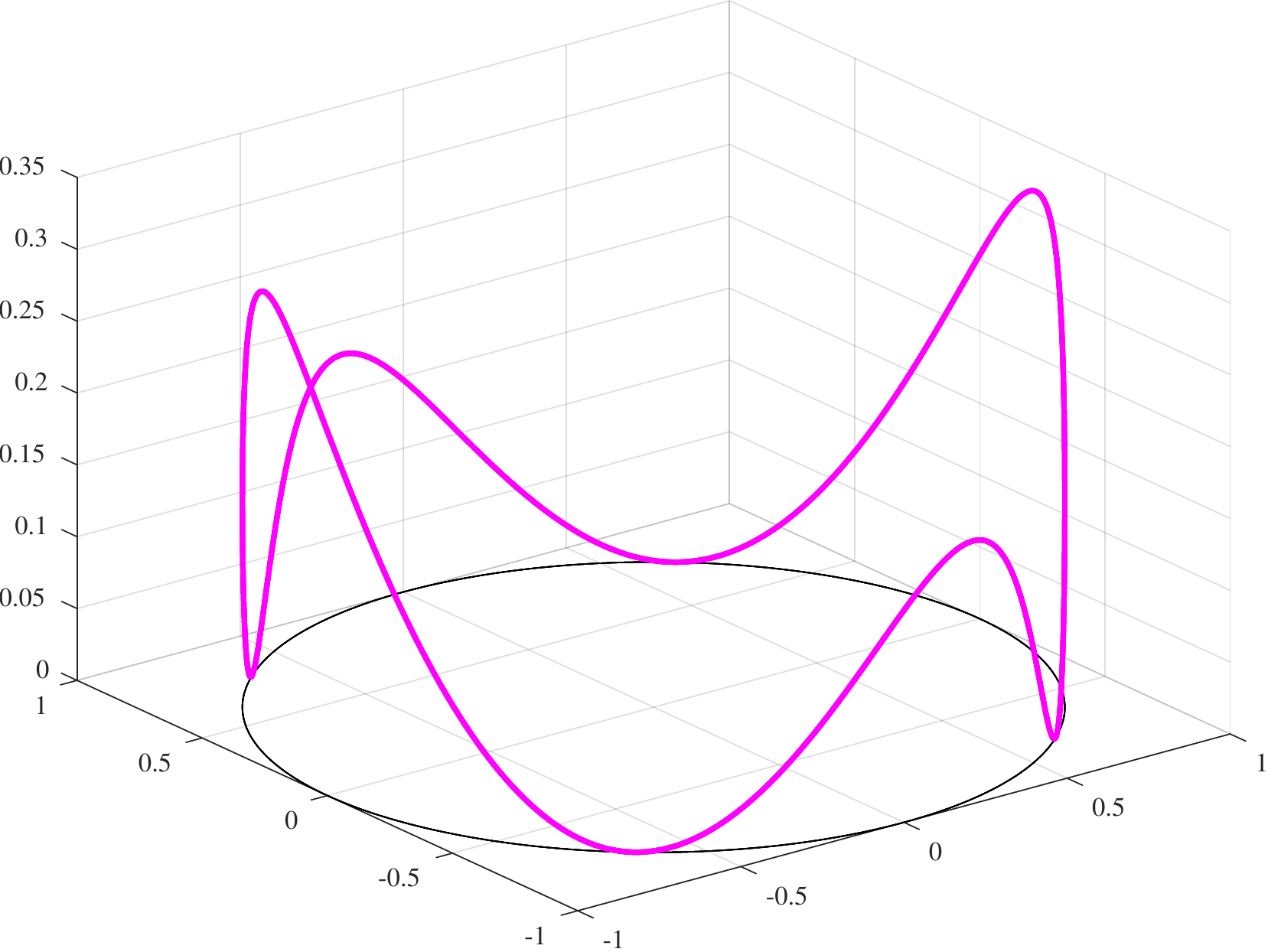}
	\caption{\textbf{Example \ref{ex: 2}} ($n=2$). \emph{Left}: The $2$ maximizers on $\mathbb{S}^1$ for both $A_{-\del_x}$ (\emph{top}) and $A_{\del_x}$ (\emph{bottom}), as indicated in \eqref{eq: max.ex.2.1}. \emph{Right}: the graph of the function $\mathbb{S}^1\ni b\mapsto\lambda_1(P(b)P(b)^\top)$ for both $A_{-\del_x}$ (\emph{top}) and $A_{\del_x}$ (\emph{bottom}), wherein we see that the maximum $\sim0.32236$ is attained at the computed maxima located on the left plots; axial symmetry of the maximizers, as well as the rotational symmetry between both functionals is also apparent.}
	\label{fig: ex2.1}
	\end{figure}
	
	\begin{figure}
	\begin{center}
	\includegraphics[scale=0.4]{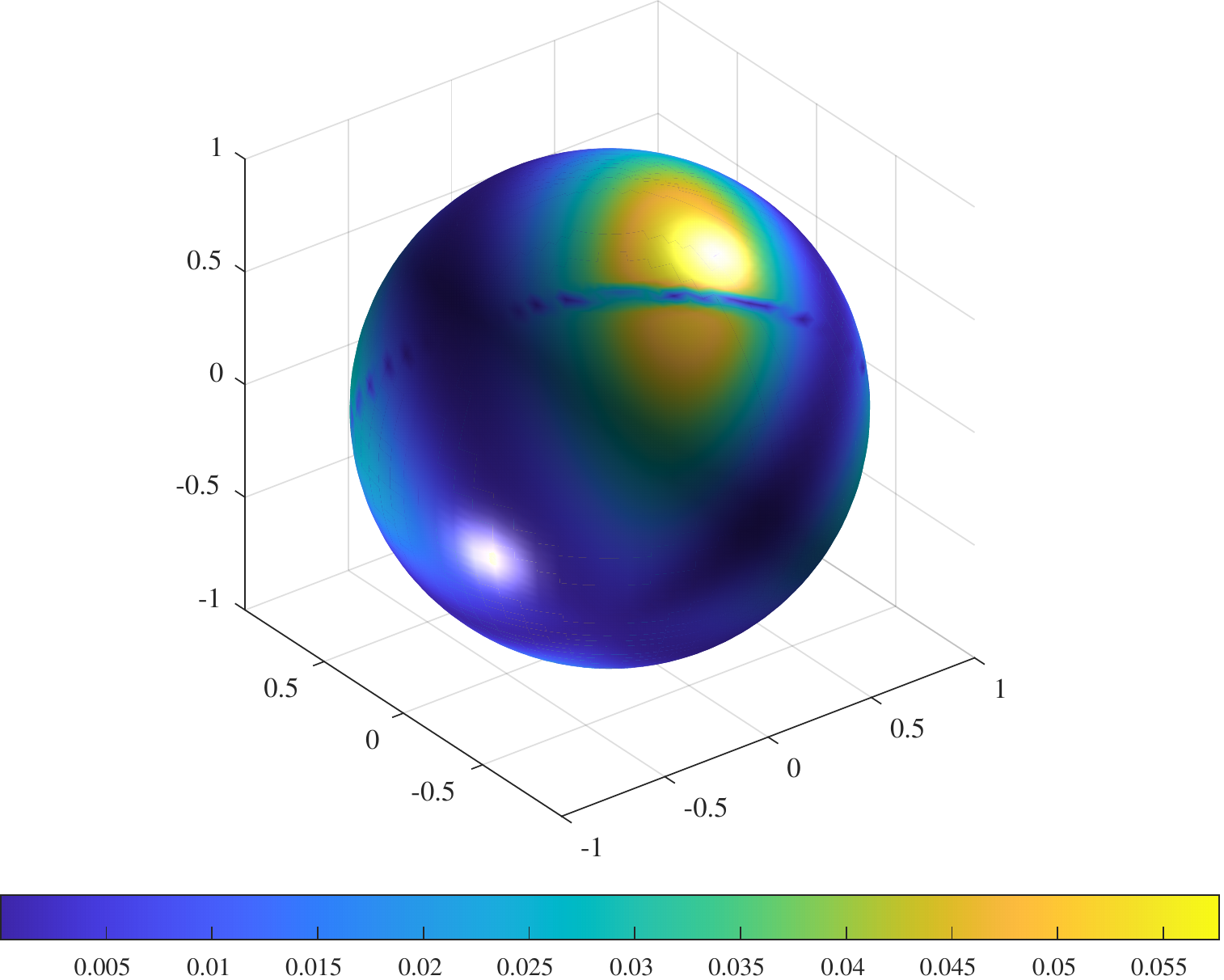}
	\includegraphics[scale=0.4]{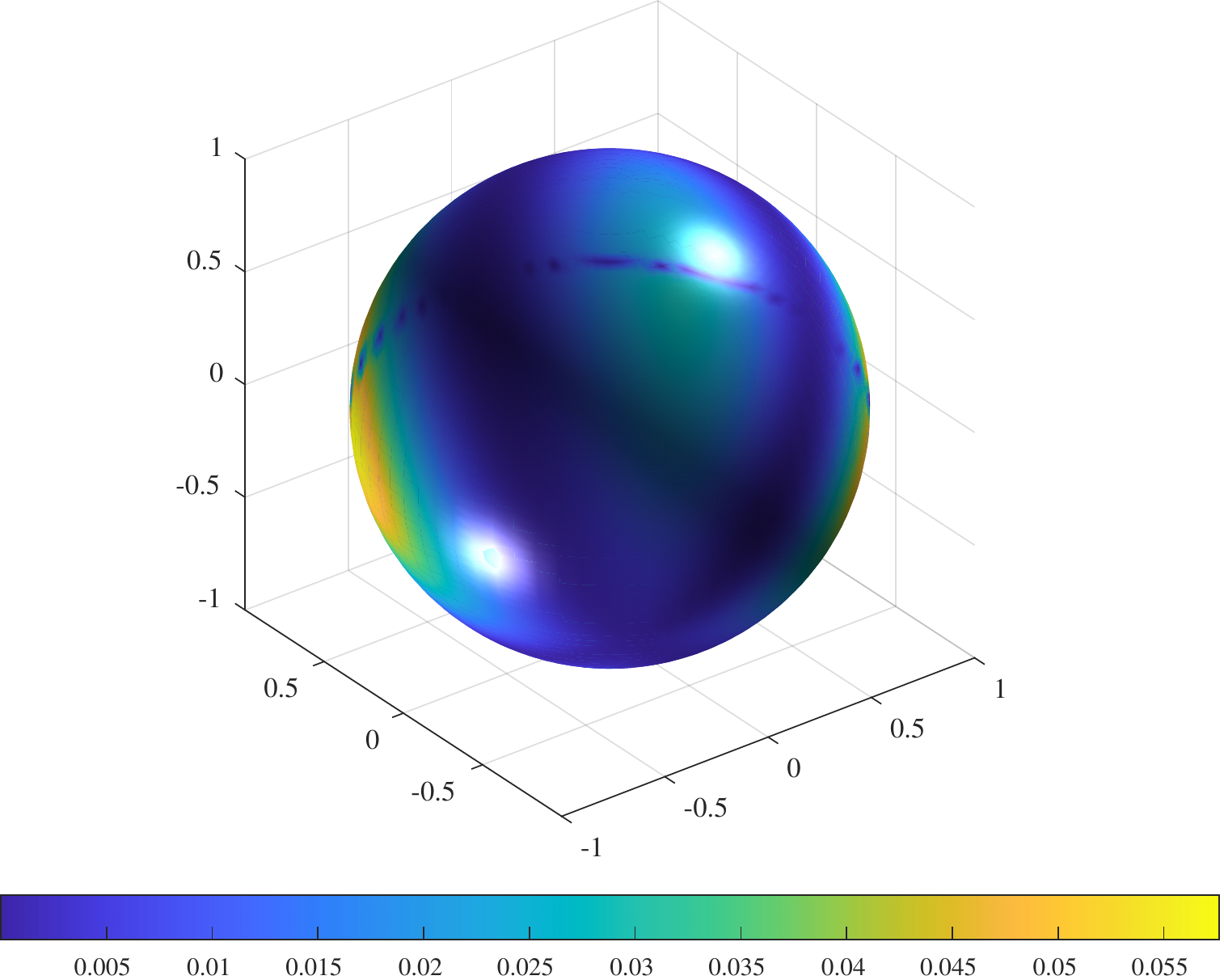}	
	\end{center}
	\caption{\textbf{Example \ref{ex: 2}} ($n=3$). The functional $b\mapsto \lambda_1(P(b)P(b)^\top)$ on $\mathbb{S}^2$. The maximizers (found in \eqref{eq: max.ex.2.2} and \eqref{eq: max.ex.2.2.1}) for both $A_{-\del_x}$ (\emph{left}) and $A_{\del_x}$ (\emph{right}) may be found in the bright yellow patches, which replicate on the opposite sides of the sphere.}
	\label{fig: ex2.2}
	\end{figure}
	
	\end{example}
	
	\section{Concluding remarks and outlook}
	
	By using the Brunovsky normal form, we discovered a reformulation of the problem consisting in finding the actuator which minimizes the controllability cost for finite dimensional linear systems with scalar controls. Such problems can be seen as, for instance, discretizations of one-dimensional lumped control problems for linear partial differential equations. We emphasize the fact that our study does not require the matrix generating the dynamics to be diagonalizable or rely on a randomization procedure of the initial data (as done in past literature in the infinite-dimensional setting).
	
	The Brunovsky reformulation provides a formulation of the control cost as a tensor product as it separates the time horizon and the controller. The resulting optimization problem reduces to the optimization of the norm of the inverse of a change of basis matrix, and allows us to stipulate the existence of minimizers (or maximizers for an equivalent variational problem), as well as non-uniqueness due to an invariance of the cost with respect to orthogonal transformations. 
	
	Let us emphasize several caveats and obstacles regarding our study, which we hope would shed some light on the possible directions of research, in view of providing a complete resolution of the optimal design problem in the deterministic case. 	
	
	\begin{itemize}
	\item The optimization of a functional which includes the inverse of a matrix is expected to not scale well with the dimension and thus possibly suffer from a curse of dimensionality. Whence, one should be wary regarding the transfer of the insights of the finite dimensional to the infinite dimensional setting.
	\smallskip 
	
	\item Even after considering the variational reformulation of the problem, which consists in maximizing the first eigenvalue of a positive-definite symmetric matrix, there are no obvious ways (to our knowledge) to solve such a mixed $\max$--$\min$ problem over a manifold such as $\mathbb{S}^{n-1}$. 
	In fact, we saw that gradient-based methods seem to fail to converge in dimensions $n\geqslant3$ -- we hence used a global optimization method based on an evolutionary algorithm, which, nonetheless, requires $\sim 8h$ to run when $n=10$ on a personal machine. We believe that a full clarification of the underlying difficulty of a numerical resolution of this problem in higher dimension, as well as the proposal of novel methods for its resolution are required.
	\end{itemize}
	
	In addition, we believe that there are a multitude of problems regarding the analysis of this problems which ought to be conducted. 
	These include the following.
	
	\subsection{Time-dependent coefficients, neural networks}
	Once all of the aforementioned problems are solved, one could look to time-dependent coefficient problems, namely for systems of the form	
	\begin{equation} \label{eq: time.dep}
	x'(t)  - A(t) x(t) = b(t) u(t) \hspace{1cm} \text{ in } (0,T).
	\end{equation}
	Note that the sparsity of $b(t)$ could also be enhanced imposing other restrictions of the form $\|b(\cdot)\|_{L^1(0,T; \R^n)} = 1$.
	
	Considering systems of the form \eqref{eq: time.dep} is particularly important in the context of \emph{deep learning} via \emph{continuous-time residual neural networks} (ResNets) (see \citep{weinan2017proposal, esteve2020large, ruiz2021neural, geshkovski2021control}), which are systems taking the form
	\begin{equation} \label{eq: resnet}
	x'(t) = \*w(t) \sigma(x(t)) + \*b(t) \hspace{1cm} \text{ in } (0,T).
	\end{equation}
	Here $\*w(t)\in\mathcal{M}_{n\times n}(\R)$ and $\*b(t)\in\R^n$ play the role of the controls, and $\sigma\in\mathrm{Lip}(\R)$. Simplifying by assuming that $\sigma=\text{Id}$, fixing $\*w(t)$, and writing $\*b(t)=bu(t)$ for $b\in\R^n$, we deduce a system of the form \eqref{eq: time.dep}. 
	
	For neural networks such as \eqref{eq: resnet}, minimizing the cost of control by means of controls which are as sparse as possible is clearly relevant for computational purposes due to the high dimensional data involved, and a linear study along with perturbation arguments could yield important insights (see \citep{yague2021sparse} for an optimal control approach to the sparsity issue). 
	There is, of course, a huge gap between the linear constant coefficient case presented above and the study of optimal controllers for ResNets. 
	But, the problems discussed above are deemed necessary in the bigger picture.
	
	\subsection{Uniqueness modulo rotations} We have seen that optimal actuators are in general not unique due to the invariance of the minimization (or maximization) problem with respect to orthogonal matrices which commute with the dynamics $A$. It would be of interest to see, at least in very particular test cases, whether a general result can be obtained characterizing the sets of optimal controllers depending on the symmetry properties of the matrix $A$. In such a case, one could perhaps deduce a uniqueness result modulo the rotated solutions. This insight is reinforced by our numerical simulations in dimensions $n=2,3$.
		
	\subsection{Non-scalar controls and PDEs} The Brunovsky normal form can also be extended to the case $m>1$, and thus $b \in \mathcal{M}_{n\times m}(\R)$. It would be of interest to see how the original problem of finding an optimal $b$  may be reformulated by means of the Brunovsky coordinates in the case $m>1$. This naturally raises the question of PDE shape design, which seems out of the scope of this particular method.
	
	\subsection{Optimization methods on manifolds}
	
	The algorithms we used need not always converge to a global maximizer lying on $\mathbb{S}^{n-1}$. The algorithm could be enforced by considering optimization methods (including gradient descent) specifically designed to variables lying on manifolds (see e.g., \citep{boumal2020introduction}\footnote{We thank Arieh Iserles for this reference.}). We leave this open to further investigation. 
	\medskip
	
	\noindent
	\textit{Acknowledgments.} We thank Yannick Privat for generally helpful comments.
	\smallskip
	
	\noindent
	{\footnotesize \textbf{Funding.} This project has received funding from the European Union's Horizon 2020 research and innovation programme under the Marie Sklodowska-Curie grant agreement No.765579-ConFlex. E.Z. has received funding from the Alexander von Humboldt-Professorship program, the European Research Council (ERC) under the European Union’s Horizon 2020 research and innovation programme (grant agreement NO. 694126-DyCon), the Transregio 154 Project “Mathematical Modeling, Simulation and Optimization Using the Example of Gas Networks” of the German DFG, grant MTM2017-92996-C2-1-R COSNET of MINECO (Spain), by the Elkartek grant KK-2020/00091 CONVADP of the Basque government and by the Air Force Office of Scientific Research (AFOSR) under Award NO: FA9550-18-1-0242.
	}

\appendix
	
	\section{Auxiliary proofs}
	
	\begin{proof}[Proof of Lemma \ref{lem: brunovski}] We only prove the first direction of the statement. 
	We split the proof in two steps.
\smallskip

\noindent \textbf{Step 1)}. Let us first assume that \eqref{eq: similar} is fulfilled for some invertible matrix $\overline{P}\in\mathcal{M}_{n\times n}(\R)$, whose columns we denote $\{\overline{f}_k\}_{k=1}^n$. 
From $b=\overline{P}\mathbf{e}_n$, we immediately deduce that $b = \overline{f}_n$, while each columns of the system $A\overline{P} = \overline{P}\mathfrak{A}$ yields 
\begin{align}
\begin{cases}
A \overline{f}_n = \overline{f}_{n-1} - a_{1}\overline{f}_n \\
A \overline{f}_{n-1} = \overline{f}_{n-2} - a_{2}\overline{f}_{n} \\
\quad \vdots\\
A \overline{f}_3 = \overline{f}_2 - a_{n-2} \overline{f}_n\\
A \overline{f}_2 = \overline{f}_1 - a_{n-1}\overline{f}_n\\
A \overline{f}_1 = -a_n \overline{f}_n.
\end{cases}
\end{align}
Here, we recall that $a_1, \ldots, a_n$ denote the coefficients of the characteristic polynomial of $A$.
The above relation can readily be rewritten to read as
\begin{align} \label{eq: def.recurs.fk}
\begin{cases}
\overline{f}_n = b,\\
A\overline{f}_{k} = \overline{f}_{k-1} - a_{n-k}\overline{f}_n, \quad \text{for all }k \in \{2, \ldots, n\},\\
A\overline{f}_1 = -a_n\overline{f}_n.
\end{cases}
\end{align}
Using the fact that \eqref{eq: def.recurs.fk} entails $\overline{f}_{k-1} = A\overline{f}_{k} + a_{n-k}b$ for $k\geqslant2$, by a brief induction argument we may further rewrite \eqref{eq: def.recurs.fk} to see that
\begin{equation} \label{eq: analy.def.fk.overload}
\overline{f}_k =
	\begin{dcases}
	b, &k=n \\
	\left(A^{n-k} + \sum_{j=1}^{n-k} a_{j} A^{n-k-j}\right) b &1\leqslant k \leqslant n-1.
	\end{dcases}
	\end{equation}
\textbf{Step 2).} Let us now define 
\begin{equation} \label{eq: P.def.proof}
P(b):= \Big[f_1 \mid \ldots \mid f_n\Big],
\end{equation}
with the columns $\{f_k\}_{k=1}^n$ of $P(b)$ being defined as in \eqref{eq: analy.def.fk.overload}. We shall prove that this $P(b)$ is invertible, and is the unique matrix such that \eqref{eq: similar} holds.

We begin by noting that
\begin{align}\label{eq: syst.fk}
\begin{aligned}
P(b)
= \quad &\Big[A^{n-1} b  \ \ A^{n-2}b \ \ A^{n-3}b \ \ \ldots  \ \ A^3b \ \  A^2b \ \ Ab \ \ b\Big]\\
+ a_1 &\Big[ A^{n-2}b \ \  A^{n-3}b \ \ A^{n-4}b \ \ \ldots \ \ A^2b  \ \ \ Ab \ \ \ b \ \ \ 0\Big]\\
+ a_2 &\Big[ A^{n-3}b \ \ A^{n-4}b \ \ A^{n-5}b \ \  \ldots \ \ \ \, Ab  \ \ \ \ b \ \ \ \ 0  \ \ \  0  \Big] \\
+ \ldots &\\
 + a_{n-3} &\Big[ \ \ A^2b  \quad \ \ Ab \qquad \ \ b \qquad \ldots \ \ \ \ \ 0  \ \ \ \  0 \ \ \ \ 0 \ \ \ 0 \, \Big]\\
 + a_{n-2} &\Big[ \ \ \ Ab  \qquad \ b \qquad \quad 0 \qquad \ldots \ \ \ \ \ 0  \ \ \ \  0 \ \ \ \ 0 \ \ \ 0 \, \Big]\\
 + a_{n-1} &\Big[ \ \ \ b  \qquad \quad 0  \qquad \quad 0 \qquad \ldots \ \ \ \ \ 0  \ \ \ \  0 \ \ \ \ 0 \ \ \ 0 \, \Big].
\end{aligned}
\end{align}
Whence, by the Kalman rank condition, $P$ has full rank and is thus invertible.
Left-multiplying the first column in \eqref{eq: syst.fk} by $A$, one obtains
\begin{align} \label{eq: cayley.hamilton}
Af_1 &= \big(A^n +a_1A^{n-1} + \ldots + a_{n-2}A^2 + a_{n-1}A\big)b = -a_n b,
\end{align}
where the rightmost equality is a consequence of the Cayley–Hamilton theorem.
Now, the definition of the columns in \eqref{eq: analy.def.fk.overload} combined with \eqref{eq: cayley.hamilton} leads us to deduce that \eqref{eq: def.recurs.fk} holds for the columns $\{f_k\}_{k=1}^n$.
Hence $AP= P\mathfrak{A}$, and one clearly also has $P\mathbf{e}_n = b$.
Thus, $P$ defined in \eqref{eq: P.def.proof} is invertible and is the unique matrix such that \eqref{eq: similar} holds.
	This concludes the proof.
	\end{proof}
	
	\begin{remark}[On the uniqueness of $P$]
	Another way to see that $P$ is the unique invertible matrix such that \eqref{eq: similar} holds is the following. Let $P_0$ be another matrix such that
	\begin{equation}
	A=P_0\mathfrak{A}P_0^{-1} \hspace{0.5cm} \text{ and } \hspace{0.5cm} b=P_0\mathbf{e}_n.
	\end{equation}
	Then, since $P$ is invertible, we may write
	\begin{equation}
	P_0 = QP
	\end{equation}
	for some matrix $Q \in \mathcal{M}_{n\times n}(\R)$.
	Thus 
	\begin{equation}
	AQP = QP\mathfrak{A} = QAP,
	\end{equation}
	so $Q$ commutes with $A$. Moreover, $QP\mathbf{e}_n = b$.
	But then
	\begin{equation}
	A^k b = A^k QP\mathbf{e}_n = QA^k P\mathbf{e}_n = QA^k b \hspace{0.5cm} \text{ for } k\geqslant 1.
	\end{equation}
	Since the vectors $A^k b$ span $\R^n$ (by virtue of the Kalman rank condition), we conclude that $Q\equiv\mathrm{Id}$.
	\end{remark}

\bibliographystyle{apalike}
	\bibliography{refs}{}

\end{document}